\documentclass{amsart}

\usepackage{amsmath,amsfonts,amssymb}
\usepackage[colorlinks=true]{hyperref}
\hypersetup{urlcolor=blue, citecolor=blue}
\usepackage{enumerate}
\usepackage{graphicx}
\usepackage{subcaption}

\newtheorem{thm}{Theorem}[section]
\newtheorem{lem}[thm]{Lemma}
\newtheorem{prop}[thm]{Proposition}
\newtheorem{coro}[thm]{Corollary}
\newtheorem{Def}[thm]{Definition}

\newtheorem{rem}[thm]{Remark}

\def\ds{\displaystyle}

\def\R{\mathbb{R}}

\def\N{\mathbb{N}}

\def\X{\mathbb{X}}
\def\be{\begin{equation}}
\def\ee{\end{equation}}
\def\g{{\bf g}}
\def\n{{\boldsymbol n}}
\def\p{\partial}
\def\grad{\boldsymbol{\nabla}}
\def\div{\grad\cdot}

\def\O{\Omega}
\def\G{\Gamma}

\def\sig{\sigma}

\def\eps{\epsilon}
\def\x{{\boldsymbol x}}
\def\d{{\rm d}}
\def\ov#1{\overline{#1}}

\def\wh#1{\widehat{#1}}
\def\wt#1{\widetilde{#1}}

\def\bvarphi{\boldsymbol{\varphi}}
\def\btau{\boldsymbol{\tau}}
\def\sign{{\text{sign}}}

\def\bu{{\boldsymbol  u}}
\def\bs{{\boldsymbol  s}}
\def\bv{{\boldsymbol v}}
\def\bz{{\boldsymbol z}}
\def\bw{{\boldsymbol w}}
\def\0{{\bf 0}}
\def\1{{\bf 1}}

\def\bbA{{\mathbb A}}

\def\bbD{{\mathbb D}}

\def\bbJ{{\mathbb J}}

\def\bbP{{\mathbb P}}

\def\bbX{{\mathbb X}}

\def\Cc{\mathcal{C}}
\def\Dd{\mathcal{D}}
\def\Ee{\mathcal{E}}
\def\Ff{\mathcal{F}}

\def\Ii{\mathcal{I}}

\def\Ll{\mathcal{L}}
\def\Mm{\mathcal{M}}

\def\Pp{\mathcal{P}}

\def\Rr{\mathcal{R}}

\def\Tt{\mathcal{T}}

\def\Vv{\mathcal{V}}

\def\Eee{\mathfrak{E}}

\def\size{{\rm size}}
\def\reg{{\rm reg}}

\def\dt{{\Delta t}}

\begin{document}

\title[Parametrization for the Richards equation]{Improving Newton's method 
performance by parametrization: the case of Richards equation}
\thanks{This work was supported by the GeoPor project funded by 
the French National Research Agency (ANR) with the grant ANR-13-JS01-0007-01 (project GEOPOR)
}
\author{Konstantin Brenner \and Cl\'ement Canc\`es}
\maketitle

\begin{abstract}
The nonlinear systems obtained by discretizing degenerate parabolic equations 
may be hard to solve, especially with Newton's method. In this paper, we apply 
to Richards equation a strategy that consists in defining a new primary unknown 
for the continuous equation in order to stabilize Newton's method by parametrizing 
the graph linking the pressure and the saturation. The resulting form of Richards 
equation is then discretized thanks to a monotone Finite Volume scheme. We 
prove the well-posedness of the numerical scheme. Then we show under 
appropriate non-degeneracy conditions on the parametrization that Newton’s 
method converges locally and quadratically. Finally, we provide numerical 
evidences of the efficiency of our approach.
\end{abstract}

\vspace{10pt}

{\small {\bf Keywords.}
Richards equation, Finite Volumes, Newton's method, parametrization
\vspace{5pt}

{\bf AMS subjects classification. }
65M22, 65M08, 76S05
}

\section{Introduction}

\subsection{Motivations and presentation of the Richards equation}

Solving numerically some nonlinear partial differential equations, for example by using 
finite elements or finite volumes,  
often amounts to the resolution of some nonlinear system of equations of the form:
\be\label{eq:syst}
\text{Find $\bu \in \R^N$ such that}\quad
\Rr(\bu) = \0_{\R^N},
\ee
where $N\in \N^\ast$ is the number of degrees of freedom and can be large. 
One of the most popular method for solving the systems of the form~\eqref{eq:syst} 
is the celebrated Newton-Raphson method. If this iterative procedure converges, 
then its limit is necessarily a solution to~\eqref{eq:syst}. However, making the Newton 
method converge is sometimes difficult and might require a great expertise. 
Nonlinear preconditioning technics have been recently developed in improve the 
performance of the Newton's method, see for instance~\cite{DGKKM_HAL, CK02}. 

Complex multiphase or unsaturated porous media flows are often modeled thanks to degenerate parabolic problems. 
We refer to~\cite{BB90} for an extensive discussion about models of porous media flows. 
For such degenerate problems, making Newton's method converge is often very difficult. 
This led to the development of several strategies to optimize the convergence properties, 
like for instance the so-called continuation-Newton method~\cite{YTA10}, or trust region based solvers~\cite{WT13}.
An alternative approach consist in solving~\eqref{eq:syst} thanks to a robust fixed point procedure with 
linear convergence speed rather that with the quadratic Newton's method (see for instance~\cite{ZBC90, JK91,JK95, PRK04}). 
Comparisons between the fixed point and the Newton's strategies are presented for instance in~\cite{LA98, BP99} (see also \cite{RPK06}).
Combinations of both technics (perform few fixed points iterations before running Newton's algorithm) 
was for instance performed in~\cite{LR16}.

Our strategy consists in reformulating the problem before applying Newton's method.
The reformulation consists in changing the primary variable in order to improve the behavior of Newton's method. 
We apply this strategy to the so called~\emph{Richards equation}~\cite{Ric31, BB90}
modeling the unsaturated flow of water within a porous medium. Extension to more complex 
models of porous media flows will be the purpose of the forthcoming contribution~\cite{BG16}.

Denote by $\O$ some open subset of $\R^d$ ($d\le 3$) representing the porous 
medium (in the sequel, $\O$ will be supposed to be polyhedral for meshing purpose), by $T>0$ a finite time horizon, 
and by $Q:=\O\times(0,T)$ the corresponding space-time cylinder.
We are interested in finding a saturation profile $\ov s:Q \to [0,1]$ 
and a water pressure $\ov p:Q \to \R$ such that 
\be\label{eq:Richards}
\p_t \ov s - \div\left(\lambda(\ov s)\left(\grad \ov p - \g\right)\right) = 0, 
\ee
where the mobility function $\lambda:[0,1] \to \R_+$ is a nondecreasing $\Cc^2$ function that 
satisfies $\lambda(s\le 0) = 0$ and $\lambda(s>0) >0$, and where $\g\in\R^d$ stands for the gravity vector. In order to 
ease the reading, we have set the porosity equal to $1$ in~\eqref{eq:Richards} and neglected the residual saturation. 
The pressure and the water content are supposed to be linked by some monotone relation 
\be\label{eq:capi}
\ov s = S(\ov p), \qquad \text{a.e. in}\;Q
\ee
where $S$ is a non-decreasing function from $\R$ to $[0,1]$. In what follows, we assume that 
$S(p)=1$ for all $p \ge 0$, that corresponds to assuming that the porous medium is water wet, 
and that $S \in L^1(\R_-)$, implying in particular that $\lim_{p\to-\infty} S(p) = 0$. 
As a consequence of the Lipschitz continuity of $\lambda$ and of the integrability of $S$ on $\R_-$, 
one has 
\be\label{eq:Kirchbound}
\lambda(S) \in L^1(\R_-).
\ee
We denote by $p_\star = \sup \{p\; | \; S(p) = 0\},$ 
with the convention that $p_\star = -\infty$ if $\{p \in \R\; | \; S(p) = 0\} = \emptyset.$

Typical behaviors of $\lambda$ and $S$ are depicted in Figure~\ref{fig:func}.
\begin{figure}[htb]
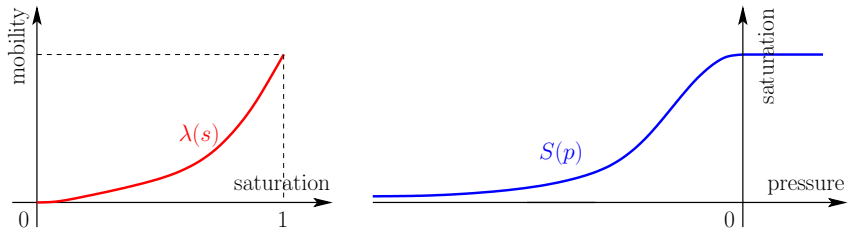

\begin{center}
\resizebox{!}{3cm}{\input{lambda.pspdftex}}
\quad
\resizebox{!}{3cm}{\input{invcapi.pspdftex}}
\end{center}
\caption{The mobility function $\lambda:[0,1] \to \R_+$ is increasing and satisfies $\lambda(0)=0$. 
The saturation function $S:\R\to [0,1]$ is non-decreasing, constant equal to $1$ on $\R_+$ and increasing 
on $(p_\star, 0)$ for some $p_\star \in [-\infty, 0)$.}
\label{fig:func}
\end{figure}

\begin{rem}
In the case where the domain $\O$ and the finite time $T$ are large, a classical hyperbolic 
scaling consisting in replacing $(\x,t)$ by $(\x/\eps, t/\eps)$ leads to the problem
$$
\p_t S^\eps(\ov p) - \div\left(\lambda(S^\eps(\ov p))\left(\grad \ov p - \g\right)\right) = 0, 
$$
where the function $S^\eps$ is deduced from $S$ by 
$$
S^\eps(p) = S(p/\eps), \qquad \forall p \in \R.
$$
Letting $\eps$ tend to $0$ leads the maximal monotone capillary pressure graph
$$
S^0(p) = \sign_+(p) =  \begin{cases}
0 & \text{if}\;p <0, \\ 
[0,1] & \text{if}\;p = 0, \\
1 & \text{if}\; p>0.
\end{cases}
$$
The Richards equation then degenerates into a hyperbolic-elliptic problem. 
Our purpose can be extended to the degenerate case even though the graph $S^0$ 
lack regularity thanks to the so-called semi-smooth Newton method (cf.~\cite{QS93}).
Moreover, because of the hyperbolic degeneracy, additional entropy criterions \emph{\`a la} 
Carrillo~\cite{Car99} are required in order to characterize the relevant solution. 
In order to simplify our purpose as much as possible, we focus on the case $\eps>0$. 
\end{rem}

The problem~\eqref{eq:Richards}--\eqref{eq:capi} is complemented by the initial condition 
\be\label{eq:init}
\ov s_{|_{t=0}} = s_0 \in L^\infty(\O;[0,1]),
\ee
and by Dirichlet boundary conditions on the pressure:
\be\label{eq:Dirichlet}
\ov p_{|_{\p\O\times(0,T)}} = p_D.
\ee
The regularity requirements on the boundary condition will be specified later on.

At least from a mathematical point of view, 
it is natural to solve~\eqref{eq:Richards} by choosing $p$ or the Kirchhoff transform $u$ (to be defined 
later on at~\eqref{eq:u})
as a primary unknown then to deduce $s = S(p) = \widetilde S(u)$. 
However, this approach lacks efficiency when one aims to solve the Richards equation numerically, 
especially for dry media, i.e, when the saturation $s$ is close to $0$. In this latter situation, it turns out 
that the Newton methods encounters difficulties to converge for solving the 
nonlinear system obtained thanks to standard implicit numerical methods 
(say $\bbP1$-Finite Elements \cite{LA98}, mixed finite elements~\cite{BP99}, or Finite Volumes~\cite{EGH99, FL01}). 
A better choice as a primary unknown in the dry regions is the saturation $s$, $p$ or $u$ being computed thanks 
to the inverse function of $S$ or $\widetilde S$ respectively. But choosing the saturation $s$ as the primary variable yields 
difficulties in the saturated regions, i.e., where $s=1$. Hence, a classical approach for 
solving numerically the Richards equation consists in applying the so-called \emph{variable switch}, 
that consists in changing the primary variable following the physical configuration (see, e.g., \cite{DP99}).

\subsection{Monotone parametrization of the graph}\label{ssec:parametrization}

The main feature of our contribution consists in parametrizing the graph $S$ in order to stabilize the Newton algorithm
without implementing the possibly complex variable switch procedure. 
This procedure is inspired from the one proposed by J. Carrillo~\cite{Car03} 
(see also~\cite{MV08} for numerical issues) to deal with hyperbolic scalar conservation 
laws with discontinuous flux w.r.t. the unknown. 
Let us introduce two continuously differentiable nondecreasing functions
$$s: (\tau_\star, \infty) \to [0,1]\quad \text{and}\quad p: (\tau_\star, \infty) \to (-\infty,\infty),$$
where $\tau_\star<0$ may be equal to $-\infty$, such that $p(0) = 0$ and 
$$
\ov s \in S(\ov p) \quad \Leftrightarrow \quad \text{there exists} \; \tau \ge \tau_\star \; \text{s.t.}\; \ov s = s(\tau) \; \text{and}\; \ov p = p(\tau).
$$
This enforces in particular that $\lim_{\tau\to \tau_\star}s(\tau) = 0$ and $\lim_{\tau \to \tau_\star}p(\tau) = p_\star$.
In the case where $\tau_\star > -\infty$, the functions $s$ and $p$ are then continuously extended into constants 
on $(-\infty,\tau_\star)$. 
It is assumed that the parametrization function $s$ satisfies 
\be\label{eq:hyp-energy-bounded}
\text{$1-s \in L^1(\R_+)$ and $s \in L^1(\R_-)$}.
\ee
The Kirchhoff transform $u:[\tau_\star, +\infty) \to \R$ is defined by 
\be\label{eq:u}
u(\tau) = 
\int_0^\tau \lambda(s(a)) p'(a) {\rm d}a, \qquad  \forall \tau \ge \tau_\star.
\ee
It follows from the integrability property \eqref{eq:Kirchbound} that 
\be\label{eq:u_star}
u_\star:= \lim_{\tau \searrow \tau_\star} u(\tau) \quad \text{is finite}.
\ee

For technical reasons, the function $u$ is artificially extended into 
a continuous onto function from $\R$ to $\R$ by setting
\be\label{eq:u-extended}
u(\tau)  = \tau-\tau_\star + u_\star, \qquad \forall \tau <\tau_\star. 
\ee
However, as it will appear later on (cf. Lemma~\ref{lem:pos}), the choice of this extension has no influence on the result.

It is assumed throughout this paper that the parametrization is not degenerated, i.e., 
$$
s'(\tau)+p'(\tau) >0 \quad \text{for a.e.}\;\tau \ge \tau_\star, 
$$
or equivalently
$$
s'(\tau)+u'(\tau) >0 \quad \text{for a.e.}\;\tau \in \R
$$
since $\lambda(s(\tau)) >0$ for all $\tau >\tau_\star$. 
Since $S$ is an absolutely continuous function, this implies in particular that $p'>0$ a.e. in $\R_+$.

Such a parametrization of the graph $S$ always exists 
but is not unique. For instance, on can choose the parametrizations defined by $p(\tau) = \tau$ 
or $u(\tau) = \tau$. 
As it will appear in the analysis carried out in the core of the paper, a convenient parametrization 
should satisfy: there exist $\alpha_\star>0$ and $\alpha^\star \ge \alpha_\star$ such that 
\be\label{eq:param-nondeg}
\alpha_\star \le \max(s'(\tau), u'(\tau)) \le \alpha^\star, \qquad \forall \tau \in \R.
\ee
Thus, we assume that~\eqref{eq:param-nondeg} holds for the analysis. 
This ensures in particular that the functions $s$ and $u$ are 
Lipschitz continuous:
\be\label{eq:su-Lip}
\|s'\|_\infty \le \alpha^\star, \qquad \|u'\|_\infty \le \alpha^\star.
\ee

For technical reasons that will appear in the analysis, 
we also assume that there exists $C>0$ such that 
\be\label{eq:u-coercive}
\tau \le C(u(\tau)+1), \qquad \forall \tau \ge 0.
\ee
It is also assumed that 
\be\label{eq:liminf-p'}
\underset{\tau \searrow \tau_\star}{\rm liminf} \; p'(\tau) >0.
\ee
This assumption is very naturally satisfied for any non-degenerate parametrization 
in the sense of~\eqref{eq:param-nondeg} of a reasonable function $S$, but unphysical counterexamples 
can be designed, enforcing us to set \eqref{eq:liminf-p'} as an assumption.

The function $s$ from $[\tau_\star, 0]$ to $[0,1]$ is nondecreasing and onto. 
Therefore, one can define the function $s^{-1}:[0,1] \to [\tau_\star, 0]$ by 
\be\label{eq:s-1}
s^{-1}(a) = \min\{ x \ge 0 \; | \; s(x) = a \}, \qquad \forall a \in [0,1].
\ee
This allows to define an initial data $\tau_0$ as $\tau_0 = s^{-1}(s_0)$ such that 
$s(\tau_0) = s_0$.

Choosing $\tau$ as the primary variable leads to the following doubly degenerate 
parabolic equation 
$$
\p_t s(\tau) - \div\Big(\lambda(s(\tau))\big(\grad p(\tau) - \g \big) \Big) = 0 \quad \text{in}\;Q.
$$
This equation 
 turns to 
\be\label{eq:Richards-tau}
\p_t s(\tau) + \div\Big( \lambda(s(\tau)) \g - \grad u(\tau)  \Big) = 0 \quad \text{in}\;Q, 
\ee
at least if $\tau \ge \tau_\star$ (this will be ensured, cf. Theorem~\ref{thm:weak}).
It is relevant to impose the boundary condition 
\be\label{eq:Dirichlet-tau}
\tau_{|_{\p\O\times(0,T)}} = p^{-1}(p_D)=:\tau_D \ge \tau_\star.
\ee
as a counterpart of~\eqref{eq:Dirichlet}. It is finally assumed that $\tau_D$ can be extended 
to the whole $\O\times(0,T)$ in a way such that 
\be\label{Dirichlet-tau-reg}
\tau_D \in C^1(\ov Q), \qquad \text{with}\quad \tau \ge \tau_\star.
\ee
The regularity required on $\tau_D$ is not optimal and can be relaxed. However,
the treatment of the boundary condition is not central in our purpose, hence we stick to~\eqref{Dirichlet-tau-reg}

\begin{Def}\label{Def:weak}
A measurable function $\tau:Q \to \R$ is said to be a weak solution to the problem~\eqref{eq:Richards-tau}, 
\eqref{eq:init}, \eqref{eq:Dirichlet-tau} if $u(\tau) - u(\tau_D) \in L^2\big((0,T);H^1_0(\O)\big)$, if $\p_t s(\tau) \in L^2\big((0,T);H^{-1}(\O)\big)$, 
and if, for all $\varphi \in C^\infty_c(\O \times [0,T); \R)$, one has
$$
\iint_Q s(\tau) \p_t \varphi \d\x\d t + \int_\O s_0\varphi(\x,0) \d\x \\
+ \iint_Q \Big(\lambda(s(\tau))  \g - \grad u(\tau)\Big)\cdot \grad \varphi \d\x\d t= 0.
$$
\end{Def}

The following statement summarizes known results about the weak solutions.
\begin{thm}\label{thm:weak}
There exists a unique weak solution $\tau:\O \to \R$ to the problem~\eqref{eq:Richards-tau}, 
\eqref{eq:init}, \eqref{eq:Dirichlet-tau} in the sense of Definition~\ref{Def:weak}. 
Moreover, $\tau \ge \tau_\star$ a.e. in $Q$,  $s(\tau) \in C([0,T];L^p(\O))$ for all $p \in [1,\infty)$, 
and, given two solutions $\tau, \wh \tau$ corresponding to two initial data $s_0$ and $\wh s_0$, 
we have
\be\label{eq:compL1}
\int_\O \left(s(\tau(\x,t)) - s(\wh\tau(\x,t))\right)^\pm \d\x \le \int_\O \left(s_0(\x) - \wh s_0(\x)\right)^\pm \d\x, \qquad 
\forall t \in [0,T].
\ee
\end{thm}

Existence of weak solutions have been proved by by Alt and Luckhaus in their seminal paper~\cite{AL83}.
We refer to~\cite{Otto96} (see also \cite{Car94,GMT94}) for extended details on the uniqueness proof and on the 
comparison principle~\eqref{eq:compL1}. The time continuity of the saturation can be proved as in~\cite{Cont_L1}.

\subsection{Outline of the paper}\label{ssec:outline}

In \S\ref{sec:scheme}, we present an implicit monotone Finite Volume scheme~\cite{EGH00}
designed for approximating the entropy solution $\tau$ of~\eqref{eq:init} and 
\eqref{eq:Richards-tau}--\eqref{eq:Dirichlet-tau}. 
First, we describe in~\S\ref{ssec:mesh} how the domains $\O$ (and then $Q$) has to be meshed. 
In particular, the mesh has to fulfill the so-called~\emph{orthogonality condition} 
so that the diffusion fluxes can be discretized using a simple two-point flux approximation~\cite{Tipi}.
The Finite Volume scheme is described in~\S\ref{ssec:scheme}. 
This scheme yields a nonlinear system of equations 
\be\label{eq:syst-n}
\Ff_n(\btau^n) = \0, \qquad \forall n \in \{1,\dots, N\}
\ee
to be solved at each time step. 
The existence and the uniqueness of the solution $\btau^n$ 
to the nonlinear system~\eqref{eq:syst-n} is proved at \S\ref{ssec:anal-scheme}. 

Once we know that the scheme~\eqref{eq:syst-n} admits one unique solution $\btau^n$, 
we discuss its effective computation thanks to Newton's method in~\S\ref{sec:Newton}. 
It is in particular proved in~\S\ref{ssec:UND} that the jacobian matrix is uniformly non-degenerate, 
so that we can use Newton-Kantorovich theorem to claim the convergence of Newton's method.  
The quantification of the error linked to the inexact resolution of the nonlinear system is performed 
in~\S\ref{ssec:inexact}.
Finally, some numerical results are presented in~\S\ref{sec:num} in order to illustrate the efficiency 
of our approach.

\section{The Finite Volume scheme}\label{sec:scheme}

\subsection{Discretization of $Q$}\label{ssec:mesh}

In this work, we only consider cylindrical discretizations of $Q$ 
that consist in discretizing space and time separately.

\subsubsection{Admissible mesh of $\O$}
The approximation of the diffusive fluxes we propose relies 
on the so-called \emph{two-point flux approximation}.
This approximation is consistent if the problem is isotropic and 
if the mesh satisfies the so-called orthogonality condition~(see e.g. \cite{Tipi}).

\begin{Def}[admissible mesh of $\O$]\label{Def:mesh}
An \emph{admissible mesh} $\left(\Tt, \Ee, \left(\x_K\right)_{K\in\Tt}\right)$ of $\O$ is given 
by a set $\Tt$ of disjointed open bounded convex subsets 
of $\O$ called \emph{control volumes}, a family $\Ee$ of subsets of $\overline\O$ 
called \emph{edges} contained 
in hyperplanes of $\R^d$ with strictly positive measure, and a family of points
$(\x_K)_{K\in\Tt}$ (the so-called \emph{cell centers}). 
It is assumed that the mesh integrates the whole $\O$, i.e., $\bigcup_{K\in\Tt} \ov K = \ov \O$.
The boundary of the control volumes are made of edges, i.e., for all $K\in\Tt$, there 
exists a subset $\Ee_K$ of $\Ee$ such that $\partial K=\bigcup_{\sig\in\Ee_K}\overline{\sig}$. 
Furthermore, $\Ee=\bigcup_{K\in\Tt}\Ee_K$. 
For any $(K,L)\in\Tt^2$ with $K\neq L$, either the $(d-1)$-dimensional Lebesgue measure of
    $\overline K\cap\overline L$ is $0$, or $\overline K\cap\overline L=\overline\sig$ for some $\sig\in\Ee$. 
    In the latter case, we write $\sig=K|L$. We denote by 
    $\Ee_{\rm int}=\left\{\sig\in\Ee, \ \exists (K,L)\in\Tt^2\ \sig=K|L\right\}$
    the set of the internal edges, and by 
    $\Ee_{\rm ext}=\{\sig\in\Ee,\ \sig\subset\partial\O\},\ \Ee_{K,\rm ext} = \Ee_K \cap \Ee_{\rm ext}$
    of the boundary edges.
    Finally, 
the family of points $(\x_K)_{K\in\Tt}$ is such that $\x_K\in K$ (for all $K\in\Tt$) and, if
$\sig=K|L$, it is assumed that the straight line $(\x_K,\x_L)$ is orthogonal to $\sig$.
For all $\sig \in \Ee_{\rm ext}$, there exists one unique cell $K$ such that $\sig \in \Ee_K$. 
Then we denote by $\x_\sig$ the projection of $\x_K$ over the hyperplane containing $\sig$, 
and we assume that $\x_\sig$ belongs to $\sig$.
\end{Def}

In what follows, we denote by $m_K$ the $d$-dimensional Lebesgue measure of the 
control volume $K\in\Tt$, and by $m_\sig$ the $(d-1)$-Lebesgue measure of the edge $\sig \in \Ee$. 
For all $\sig\in\Ee_K$, we denote by $d_{K,\sig} = d(\x_K,\x_\sig)$. Since $\sig=K|L$ is supposed to 
be orthogonal to $\x_K-\x_L$, then $d(\x_K,\x_L) = d_{K,\sig}+d_{L,\sig}=:d_{\sig}$.
We define the 
\emph{transmissibilities} $\left(A_\sig\right)_{\sig \in \Ee}$ by  
$$A_\sig = 
\begin{cases}
\frac{m_\sig}{d_{\sig}} &\text{if} \; \sig = K|L \in \Ee_{\rm int},\\
\frac{ m_\sig }{d_{K,\sig}}&\text{if} \; \sig \in \Ee_K\cap \Ee_{\rm ext}.
\end{cases}
$$
The space of the degrees of freedom (including those prescribed by the boundary condition) is
$$
\X_\Tt = \left\{\bv = \left(v_K, v_\sig \right)_{K\in\Tt, \sig \in \Ee_{\rm ext}} \right\} \simeq \R^{\#\Tt + \#\Ee_{\rm ext}}, 
$$
while the interior degrees of freedom (for which a nonlinear system has to be solved) are described by the space 
$$
\X_{\Tt, {\rm int}} = \left\{\bv = \left(v_K \right)_{K\in\Tt} \right\} \simeq \R^{\#\Tt}.
$$

\subsubsection{Time and space-time discretizations}

\begin{Def}[Time discretizations of $(0,T)$]\label{Def:time-disc}
A {\emph time discretization}
of $(0,T)$ is given by an integer value $N$ and a sequence of real values 
$0 = t^0 < t^1 < \ldots < t^N = T$. For all $n\in\{1,\dots,
 N\}$ the time step is defined by $\dt^n= t^n - t^{n-1}$.
\end{Def}

\begin{Def}[Space-time discretizations of $Q$]\label{Def:space-time-disc}
A space-time discretization $\Dd$ of $Q$ is a family
$$\Dd=(\Tt,\Ee,(\x_K)_{K\in\Tt},(t^n)_{n\in\{0,\dots,N\}}),$$
where $(\Tt,\Ee,(\x_K)_{K\in\Tt})$ is an admissible mesh of $\O$ in the sense of Definition~\ref{Def:mesh}
and $(N,(t^n)_{n\in\{0,\dots,N\}})$ is a discretization of $(0,T)$ in the sense of Definition~\ref{Def:time-disc}.
\end{Def}
The spaces of the degrees of freedom are defined by
\begin{equation}\label{eq:X_Dd}
\X_{\Dd} = \left\{ \bv = \left(v_K^n, v_\sig^n\right)_{K\in \Tt, \sig \in \Ee_{\rm ext}, 1 \le n \le N} \right\} \simeq \R^{(\#\Tt+\#\Ee)  \times N}.
\end{equation}  
and 
\begin{equation}\label{eq:X_Dd}
\X_{\Dd, \rm int} = \left\{ \bv = \left(v_K^n\right)_{K\in \Tt, 1 \le n \le N} \right\} \simeq \R^{\#\Tt  \times N}.
\end{equation}  
Let $\bv = \left(v_K^n\right)_{K,n} \in \X_\Dd$, then we denote by $\bv^n = \left(v_K^n\right)_{K} \in \X_\Tt$ 
for $n \in \{1,\dots, N\}$.

\subsubsection{Reconstruction operators}\label{sssec:operators}

Following the approach proposed in~\cite{DEGH13}, we introduce 
reconstruction operators. First, we define the linear operator
$\pi_\Tt: \X_\Tt \to L^\infty(\O)$ by 
$$
\pi_\Tt \bv(\x) = v_K \; \text{ if } \x\in K, \qquad \forall \bv = {(v_K, v_\sig)}_{K\in\Tt, \sig \in \Ee_{\rm ext}}.
$$
It is extended into the time-and-space reconstruction linear operator $\pi_\Dd: \left(\X_\Tt\right)^N \to L^\infty(Q)$
by setting
$$
\pi_\Dd \bv(\x,t) = v_K^n \; \text{ if } (\x,t)\in K\times(t^{n-1},t^n], 
\qquad \forall \bv = {(v_K^n, v_\sig^n)}_{K\in\Tt, \sig \in \Ee_{\rm ext}}^{1\le n \le N} \in \left(\X_\Tt\right)^N.
$$

The study to be performed also requires the introduction of a so-called 
discrete gradient. We will remain sloppy about the construction of the discrete 
gradient. We only highlight the properties we will use in the sequel.

\begin{lem}\label{lem:reconstruct}
Let $\Tt$ be an admissible discretization of $\O$ in the sense of Definition~\ref{Def:mesh}. 
There exists a linear operator $\grad_\Tt:\X_\Tt \to L^\infty(\O)^d$ such that 
for all $\bv = \left(v_K, v_\sig\right)_{K,\sig}$ and $\bw = \left(w_K, w_\sig\right)_{K,\sig}$ in $\X_\Tt$, one has 
\begin{align}\label{eq:prod-scal}
\int_\O \grad_\Tt \bv \cdot \grad_\Tt \bw\d\x = & \sum_{\sig = K|L \in \Ee_{\rm int}} 
A_\sig (v_K - v_L)(w_K - w_L) \\
&  + \sum_{K\in\Tt} \sum_{\sig \in \Ee_{K,\rm ext}} 
A_\sig (v_K - v_{\sig}) (w_K-w_\sig). \nonumber
\end{align}
Moreover, let $\left(\Tt_m\right)_{m\ge 1}$ be a sequence of admissible discretizations of $\O$ 
in the sense of Definition~\ref{Def:mesh} such that $\size(\Tt_m)$ tends to $0$ while $\reg({\Tt_m})$ 
remains bounded as $m$ tends to $\infty$, and let $\left(\bv_m\right)_{m\ge 1}$ be a family such that 
$\bv_m \in \X_{\Tt_m}$ for all $m\ge1$ and 
$$\| \pi_\Tt \bv_m \|_{L^2(\O)} +  {\|\grad_{\Tt_m} \bv_m \|}_{L^2(\O)^d} \le C, \qquad \forall m \ge 1,$$
then 
there exists $v \in H^1(\O)$ such that, up to an unlabeled subsequence, one has 
\be\label{eq:compact}
\pi_{\Tt_m} \bv_m \underset{m\to\infty}\longrightarrow v \quad\text{in}\; L^2(\O),
\ee
and 
\be\label{eq:consist}
\grad_{\Tt_m} \bv_m \underset{m\to\infty}\longrightarrow \grad v \quad\text{weakly in}\; L^2(\O)^d.
\ee
Additionally, if $\varphi \in C^\infty_c(\O)$ is discretized into 
$\bvarphi_m= \left(\varphi_K\right)_{K\in\Tt_m} \in \X_{\Tt_m}$ by setting 
$$
\varphi_K  = \frac1{m_K} \int_K \varphi(\x) \d\x, \qquad \forall K \in \Tt_m, 
$$
then $\grad_{\Tt_m} \bvarphi_m$ converges strongly in $L^2(\O)^d$ towards $\grad\varphi$ as $m$ tends to $+\infty$.
\end{lem}

Note that the discrete gradient reconstruction used in~\cite[\S4.1]{ACM} 
(which is the usual one for two-point flux approximations of diffusion operators)
does not meet the requirements of Lemma~\ref{lem:reconstruct} since \eqref{eq:prod-scal}
is not fulfilled. However, a reconstruction as prescribed by Lemma~\ref{lem:reconstruct} 
can be obtained as a particular case of the so-called SUSHI scheme
on ``super-admissible meshes" (cf.~\cite[Lemma 2.1]{EGH10}).

\smallskip

Finally, the reconstruction operators $\pi_\Tt:\X_\Tt \to L^\infty(\O)$ and $\grad_\Tt:\X_\Tt \to L^\infty(\O)^d$ 
are extended to the space-times framework into 
$\pi_\Dd:\X_\Dd \to L^\infty(Q)$ and 
$\grad_\Dd:\X_\Dd \to L^\infty(Q)^d$ defined for all 
$\bv = \left(\bv^n\right)_{1\le n\le N} \in \X_\Dd$ by 
\be\label{eq:piDdgradDd}
\pi_\Dd \bv(\cdot,t)  = \pi_\Tt \bv^n, \qquad
\grad_\Dd \bv(\cdot,t)  = \grad_\Tt \bv^n, \qquad \forall t \in (t^{n-1},t^n].
\ee

\subsection{The implicit finite volume scheme}\label{ssec:scheme}

The initial data $s_0$ is discretized into $\bs^0 = \left(s_K^0\right)_{K\in\Tt} \in \X_{\Tt, \rm int}$ by setting 
\be\label{eq:sK0}
s_K^0 = \frac1{m_K}\int_K s_0(\x) \d\x, \qquad \forall K\in\Tt. 
\ee
Notice that $0 \le s_K^0 \le 1$ since $0 \le s_0 \le 1$. 
We define $\btau^0=\left(\tau_K^0\right)_K \in\X_{\Tt, \rm int}$ as 
\be\label{eq:btau0}
\btau^0 = s^{-1}(\bs^0), \ee
so that 
$$
s(\tau_K^0) = s_K^0, \qquad \forall K \in \Tt.
$$
The boundary condition $\tau_D \in C^1(\ov Q)$ is discretized by setting for all $n \in \{0,\dots, N\}$ 
\be\label{eq:tauDdisc}
\tau_{D,\sig}^n= \tau_D(\x_\sig, t^n) , \quad \forall \sig \in \Ee_{\rm ext}, \quad \text{and}\quad
\tau_{D,K}^n= \tau_D(\x_K, t^n) , \quad \forall K \in \Tt.
\ee
Then we denote by $\btau_D^n  = \left(\tau_{D,K}^n, \tau_{D,\sig}^n\right)_{K\in \Tt, \sig \in \Ee_{\rm ext}} \in \X_\Tt$, and by 
$\btau_D = \left(\btau_D^n\right)_{1\le n\le N} \in \X_\Dd$. 
It follows from Formula~\eqref{eq:prod-scal} and from the regularity of $\tau_D$ that 
\begin{multline}\label{eq:tauDdisc-reg}
\int_\O |\grad_\Tt \btau_D^n|^2 \d\x = \sum_{\sig =K|L\in \Ee_{\rm int}} A_\sig (\tau_{D,K}^n - \tau_{D,L}^n)^2 + 
\sum_{\sig \in \Ee_{\rm ext}} A_\sig  (\tau_{D,K}^n - \tau_{D,\sig}^n)^2 \\
\le \|\grad \tau_D\|_\infty^2 \left(  \sum_{\sig =K|L\in \Ee_{\rm int}} m_\sig d_{\sig} + 
\sum_{\sig \in \Ee_{\rm ext}}  m_\sig d_{K,\sig}\right) = d m_\O \|\grad \tau_D\|_\infty^2.
\end{multline}

Let $n\ge1$. 
Assume that the state $\btau^{n-1}=\left(\tau_K^{n-1}\right)_K \in \X_\Tt$ is known.
The implicit finite volume scheme is obtained by writing the local conservation of 
the volume of each fluid on the control volumes, i.e., 
\be\label{eq:scheme}
\frac{s(\tau_K^n) - s(\tau_K^{n-1})}{\dt^n} m_K 
+ \sum_{\sig \in \Ee_K} F_{K,\sig}^n = 0, \qquad \forall K \in \Tt, 
\ee
where $F_{K,\sig}^n$ denotes the outward w.r.t. $K$ flux across the edge $\sig$ at time step $t^n$. 
Denote by $\n_{K,\sig}$ the outward w.r.t. $K$ normal to $\sig$, and by 
$g_{K,\sig} = \g\cdot\n_{K,\sig}$ for all $\sig \in \Ee_K$ and all $K\in \Tt$.
Denote by 
$$
\tau_{K,\sig}^n = \begin{cases}
\tau_L^n & \text{if}\; \sig = K|L \in \Ee_{\rm int}, \\
\tau_{D,\sig}^n & \text{if}\; \sig \in \Ee_{K,\rm ext},  
\end{cases}
$$
then
the fluxes $F_{K,\sig}^n$ across $\sig \in \Ee_{K}$ is defined by 
\be\label{eq:FKsig}
F_{K,\sig}^n = m_\sig \left( \lambda(s(\tau_K^n)) g_{K,\sig}^+ - \lambda(s(\tau_{K,\sig}^n)) g_{K,\sig}^-\right)
+ A_\sig \left( u(\tau_K^n) - u(\tau_{K,\sig}^n)\right).
\ee
Note in particular that the scheme is locally conservative, i.e., 
$$
F_{K,\sig}^n + F_{L,\sig}^n = 0, \qquad \forall \sig = K|L \in \Ee_{\rm int}.
$$
Combining~\eqref{eq:scheme}--\eqref{eq:FKsig}, the scheme can be rewritten in a condensed form as 
\be\label{eq:scheme2}
\Rr_K\left(\tau_K^n, \tau_K^{n-1}, \left(\tau_L^n\right)_{L\neq K}, \left(\tau_{D,\sig}^n\right)_{\sig \in \Ee_{K,\rm ext}}\right)=0, \qquad \forall K \in \Tt, 
\ee
where $\Rr_K$ is nondecreasing w.r.t. its first argument and nonincreasing w.r.t. the others 
thanks to the monotonicity of the functions $\lambda, s$ and $u$. 

It is worth noticing that
\be\label{eq:div-nulle}
\sum_{\sig \in \Ee_K} m_\sig g_{K,\sig} = 0, \qquad \forall K \in \Tt. 
\ee
Therefore, the convective flux balance can be reformulated, yielding
\begin{align}\label{eq:balance-flux}
\sum_{\sig \in \Ee_K} F_{K,\sig}^n = & 
\sum_{\sig \in \Ee_KL} m_\sig g_{K,\sig}^- \left(\lambda(s(\tau_K^n)) - \lambda(s(\tau_{K,\sig}^n))\right) \\
&		+ \sum_{\sig = K|L}  A_\sig (u(\tau_K^n) - u(\tau_{K,\sig}^n)) 
		, \qquad \forall K \in \Tt. \nonumber
\end{align}

\subsection{Existence and uniqueness of the solution to the scheme}\label{ssec:anal-scheme}

In this section, we analyze the system~\eqref{eq:scheme2} obtained for a fixed admissible discretization 
$\Dd$ of $Q$.
In what follows, we denote by 
$$
a\top b = \max(a,b) \quad \text{and}\quad a\bot b= \min(a,b), \qquad \forall (a,b) \in \R^2.
$$

The following Lemma is a discrete counterpart of the $L^1$-contraction principle~\eqref{eq:compL1} 
on the exact solution. 
\begin{lem}\label{lem:contract}
Let $\btau^{n-1}$ and $\wh \btau^{n-1}$ be two elements of $\X_{\Tt, \rm int}$, and let 
$\btau^{n}$ and $\wh \btau^{n}$ in $\X_{\Tt, \rm int}$ be two corresponding solutions, then 
\be\label{eq:contract-disc}
\int_\O \left| \pi_\Tt s(\btau^n) - \pi_\Tt s(\wh \btau^n) \right| \d\x \le \int_\O \left| \pi_\Tt s(\btau^{n-1}) - \pi_\Tt s(\wh \btau^{n-1}) \right| \d\x.
\ee
\end{lem}
\begin{proof}
It follows from the monotonicity of $\Rr_K$ that 
\begin{eqnarray*}
\Rr_K\left(\tau_K^n, \tau_K^{n-1}\top\wh\tau_K^{n-1}, \left(\tau_L^n \top\wh\tau_L^n\right)_{L\neq K}, \left(\tau_{D,\sig}\right)_{\sig \in \Ee_{K,\rm ext}}\right) \le 0, \\
\Rr_K\left(\wh\tau_K^n, \tau_K^{n-1}\top\wh\tau_K^{n-1}, \left(\tau_L^n \top\wh\tau_L^n\right)_{L\neq K} , \left(\tau_{D,\sig}\right)_{\sig \in \Ee_{K,\rm ext}}\right) \le 0.
\end{eqnarray*}
Since $\tau_K^n\top \wh\tau_K^n$ is either equal to $\tau_K^n$ or to $\wh\tau_K^n$, one has 
\be\label{eq:comp-top}
\Rr_K\left(\tau_K^n\top\wh\tau_K^n, \tau_K^{n-1}\top\wh\tau_K^{n-1}, 
				\left(\tau_L^n \top\wh\tau_L^n\right)_{L\neq K}, \left(\tau_{D,\sig}\right)_{\sig \in \Ee_{K,\rm ext}}\right) \le 0.
\ee
Similar calculations lead to 
\be\label{eq:comp-bot}
\Rr_K\left(\tau_K^n\bot\wh\tau_K^n, \tau_K^{n-1}\bot\wh\tau_K^{n-1}, 
				\left(\tau_L^n \bot\wh\tau_L^n\right)_{L\neq K}, \left(\tau_{D,\sig}\right)_{\sig \in \Ee_{K,\rm ext}}\right) \ge 0.
\ee
Summing~\eqref{eq:comp-top} with~\eqref{eq:comp-bot} and over $K\in\Tt$ yields~\eqref{eq:contract-disc}.
\end{proof}

\begin{lem}\label{lem:unique}
Given $\btau^{n-1}\in \X_\Tt$, then there exists at most one solution $\btau^n\in \X_\Tt$ to the 
scheme~\eqref{eq:scheme}--\eqref{eq:FKsig}. 
\end{lem}
\begin{proof}
As a direct consequence of Lemma~\ref{lem:contract}, $s(\tau_K^n) = s(\wh\tau_K^n)$ for all $K\in\Tt$.
Subtracting the system yielding $\wh\btau^{n}$ to the one corresponding to $\btau^{n}$ leads to 
$$
\sum_{\sig = K|L \in \Ee_K} A_\sig\left(w_K^n - w_L^n\right) + \sum_{\sig\in\Ee_{\rm ext}\cap \Ee_K} A_\sig w_K^n = 0, 
\qquad \forall K\in\Tt,
$$
where we have set $w_K^n = u(\tau_K^n) - u(\wh\tau_K^n)$. It follows from classical 
arguments that $\bw^n=\left(w_K^n\right)_K = \0_{\X_\Tt}$. 
Bearing the nondegeneracy condition~\eqref{eq:param-nondeg} in mind, we get that $\btau^n = \wh\btau^n$.
\end{proof}

The following lemma shows that the solution $\btau^n$ to the scheme is always greater than $\tau^\star$. 
Therefore, the extension~\eqref{eq:u-extended} we chose for the function $u$ does not affect the result.

\begin{lem}\label{lem:pos}
Let $\btau^n\in\X_\Tt$ be the solution to~\eqref{eq:scheme}--\eqref{eq:FKsig}, then $\tau_K^n \ge \tau_\star$ for all 
$K \in \Tt$.
\end{lem}
\begin{proof}
There is nothing to prove if $\tau_\star = -\infty$, hence let us assume that $\tau_\star$ is finite. 
Let $K$ be a cell such that $\tau_K^n \le \tau_{K,\sig}^n$ for all $\sig \in \Ee_K$, and assume that $\tau_K^n <\tau_\star$.
Since $s(\tau_K^n)=0$ and $\lambda(s(\tau_K^n))=0$, one gets that 
$$
\sum_{\sig \in \Ee_{K}} A_\sig\left(u(\tau_K^n) -u(\tau_{K,\sig}^n)\right) \ge  0.
$$
The extension~\eqref{eq:u-extended} of $u$ ensures that $u(\tau_K^n)<u(\tau_\star) \le u(\tau_{D,\sig}^n)$. The left-hand side 
of the above relation is therefore negative, hence a contradiction with the assumption $\tau_K^n <\tau_\star$.
\end{proof}

Let $n \in \{0,\dots, N\}$, then define 
$$
e_K^n(\tau) = \int_{\tau_{D,K}^n}^\tau (a-\tau_{D,K}^n) s'(a)\d a =\int_{\tau_{D,K}^n}^\tau \big(s(\tau) - s(a)\big) \d a \ge0, \quad \forall \tau \in \R, 
$$
and 
$\Eee^n :\X_\Tt \to \R_+$ by 
$$
\Eee^n(\btau) = \sum_{K\in\Tt} e_K^n (\tau_K) m_K, \qquad \forall \btau = \left(\tau_K\right)_{K\in\Tt}.
$$
It is easy to verify (see e.g.~\cite{CP12}) that 
\be\label{eq:borne-Eee}
0 \le \Eee^n(\btau) \le m_\O \left(\|1-s\|_{L^1(\R_+)} + \|s\|_{L^1(\R_-)}\right), \qquad \forall \btau \in \X_\Tt. 
\ee
Moreover, it follows from the $C^1$ regularity of $\tau_D$ and from the fact that $0 \le s \le 1$ that 
\be\label{eq:diffEee}
|\Eee^n(\btau) - \Eee^{n-1}(\btau)| \le \dt^n m_\O \|\p_t \tau_D\|_{\infty}.
\ee

We define the Lipschitz continuous function $\xi:\R_+\to \R$ by 
\be\label{eq:xi}
\xi(\tau) = \int_0^\tau \sqrt{\lambda(s(a)) p'(a)} \d a = \int_0^\tau \sqrt{u'(a)} \d a, \qquad \forall \tau \in \R, 
\ee
then it follows from Cauchy-Schwartz inequality that 
\be\label{eq:xi2}
(a-b)(u(a) - u(b)) \ge (\xi(a) - \xi(b))^2, \qquad \forall (a,b) \in \R^2.
\ee
Moreover, the Lipschitz continuity of $\xi$ implies that 
$$
u(\tau) \le \|\xi'\|_\infty \xi(\tau), \qquad \forall \tau \in \R_+, 
$$
hence it follows from Assumption~\eqref{eq:u-coercive} that 
\be\label{eq:xi-coercive}
\tau \le C \left(\xi(\tau) + 1\right), \qquad \forall \tau \in \R_+, 
\ee
then, in particular, one has 
\be\label{eq:xi-coercive2}
\lim_{\tau\to\infty} \xi(\tau) = +\infty.
\ee

\begin{lem}\label{lem:wBV}
Let $\btau^n$ be a solution to the scheme~\eqref{eq:scheme}--\eqref{eq:FKsig}. 
Then there exists $C_1$ depending only on $\O$ and $\tau_D$ such that the following estimate holds:
\be\label{eq:energy}
\Eee^n(\btau^n) + \dt^n \left( C_1 + 
\frac12\int_\O |\grad_\Tt \xi(\btau)|^2 \d\x
 \right)
\le \Eee^{n-1}(\btau^{n-1}).
\ee
\end{lem}
\begin{proof}
We multiply the equation~\eqref{eq:scheme} by $\dt^n (\tau_K^n-\tau_{D,K}^n)$ and sum over $K\in \Tt$. 
Using~\eqref{eq:balance-flux}, this provides 
\be\label{eq:T123}
T_1 +  \dt^n \big(T_2 +  T_3\big) = 0, 
\ee
where
\begin{align*}
T_1 = &\ds  \sum_{K\in\Tt} \left(s(\tau_K^n) - s(\tau_K^{n-1})\right) (\tau_K^n - \tau_{D,K}^n) m_K, \\
T_2 = &\ds   \sum_{K\in\Tt} \sum_{\sig\in\Ee_K} m_\sig g_{K,\sig}^- 
		\left( \lambda(s(\tau_K^n)) - \lambda(s(\tau_{K,\sig}^n)) \right) (\tau_K^n - \tau_{D,K}^n), \\
T_3 = & \int_\O \grad_\Tt u(\btau^n) \cdot \grad_\Tt (\btau^n - \btau_D^n) \d\x.
\end{align*}

It follows from the convexity of $e\circ s^{-1}$ (see e.g.~\cite[Proposition 3.7]{CG16_MCOM}) that 
$$T_1 \ge \Eee^n(\btau^n) - \Eee^n(\btau^{n-1}).
$$
Then thanks to~\eqref{eq:diffEee}, one gets that 
\be\label{eq:T1}
T_1 \ge \Eee^n(\btau^n) - \Eee^{n-1}(\btau^{n-1}) - \dt^n m_\O \|\p_t \tau_D\|_\infty.
\ee

The term $T_2$ can be estimated following the path of \cite{EGGH98}. 
Denote by $\Phi:\R\to\R_+$ the function defined by 
$$
\Phi(\tau) = \int_0^\tau a s'(a) \lambda'(s(a)) \d a, \qquad \forall \tau \in \R.
$$ 
Since $s(0) = 1$, one has $s'(a) = 0$ for all $a \ge 0$, and thus 
$$
\Phi(\tau) = 0 \quad \text{if}\quad \tau >0, \quad \text{and} \quad \Phi'(\tau) \le 0, \quad \forall \tau \le 0.
$$
Moreover, since $s \in L^1(\R_-)$, one has 
$$
|\tau| \lambda(s(\tau))  \le \|\lambda'\|_\infty |\tau|s(\tau) \underset{\tau \searrow \tau_\star}{\longrightarrow} 0.
$$
Therefore, for all $\tau \le 0$, one has 
$$
\Phi(\tau) = \tau \lambda(s(\tau)) + \int_\tau^0 \lambda(s(a)) \d a \le \int_\tau^0 \lambda(s(a)) \d a + C.
$$
Thanks to~\eqref{eq:u_star}  and to Assumption~\eqref{eq:liminf-p'}, one has 
$
\int_\tau^0 \lambda(s(a)) \d a \le C, 
$
hence $\Phi$ is bounded.
Simple calculations 
show that for all $(a,b) \in \R^2$, one has
$$
b \left( \lambda(s(b)) - \lambda(s(a)) \right) =  \Phi(b) - \Phi(a)  + \int_a^b \big(\lambda(s(r))-\lambda(s(a))\big)\d r
 \ge \Phi(b) - \Phi(a).
$$
Rewriting 
\be\label{eq:T21}
T_2 = T_{21} + T_{22} 
\ee
with 
\begin{align*}
T_{21} = &  \sum_{K\in\Tt} \sum_{\sig\in\Ee_K} m_\sig g_{K,\sig}^- 
		\left( \lambda(s(\tau_K^n)) - \lambda(s(\tau_{K,\sig}^n)) \right) \tau_K^n, \\
T_{22} = & -   \sum_{K\in\Tt} \sum_{\sig\in\Ee_K} m_\sig g_{K,\sig}^- 
		\left( \lambda(s(\tau_K^n)) - \lambda(s(\tau_{K,\sig}^n)) \right) \tau_{D,K}^n, 
\end{align*}
we get that 
\begin{align*}
T_{21} \ge&  \sum_{K\in\Tt} \sum_{\sig\in\Ee_K} m_\sig g_{K,\sig}^- \left(\Phi(s(\tau_K^n))  - \Phi(s(\tau_{D,K}^n))\right) \\
\ge & \sum_{K\in\Tt} \sum_{\sig\in\Ee_{K,\rm ext}} m_\sig \left(g_{K,\sig}^+ \Phi(s(\tau_K^n))  - g_{K,\sig}^- \Phi(s(\tau_{D,K}^n))\right).
\end{align*}
Using the boundedness of $\Phi$, we get that 
\be\label{eq:T21}
T_{21} \ge-  m_{\p\O} |\g| \|\Phi\|_\infty.
\ee
On the other hand, a classical reorganization of the term $T_{22}$ provides 
\begin{multline*}
T_{22}  = \sum_{\sig = K|L \in \Ee_{\rm int}} m_\sig \left(g_{K,\sig}^+ \lambda(s(\tau_K^n)) - g_{L,\sig}^+ \lambda(s(\tau_L^n))\right)
(\tau_{D,K}^n - \tau_{D,L}^n) \\
- \sum_{K\in \Tt} \sum_{\sig \in \Ee_{K,\rm ext}} m_\sig  \left(g_{K,\sig}^+ \lambda(s(\tau_K^n)) - g_{K,\sig}^- \lambda(s(\tau_{D,\sig}^n))\right) \tau_{D,K}^n.
\end{multline*}
Therefore, it follows from the regularity of $\tau_D$ and from the boundedness of $\lambda$ that 
\be\label{eq:T22}
T_{22} \ge - |\g| \|\lambda\|_\infty \left(d  m_\O \|\grad \tau_D\|_\infty - m_{\p\O} \|\tau_D\|_\infty\right).
\ee

Finally, it results from~\eqref{eq:xi2}, from the relation $u'= (\xi')^2$, and from~\eqref{eq:tauDdisc-reg} that
\begin{multline*}
T_3 \ge  \|\grad_\Tt \xi(\btau^n)\|^2_{L^2(\O)^d} - \|\grad_\Tt u(\btau^n)\|_{L^2(\O)^d}  \|\grad_\Tt \btau_D^n\|_{L^2(\O)^d} \\
\ge \|\grad_\Tt \xi(\btau^n)\|^2_{L^2(\O)^d} - \sqrt{\|u'\|_\infty} \|\grad_\Tt \xi(\btau^n)\|_{L^2(\O)^d}  \|\grad_\Tt \btau_D^n\|_{L^2(\O)^d}.
\end{multline*}
Therefore, it follows from~\eqref{eq:tauDdisc-reg} that 
\be\label{eq:T3}
T_3 \ge \frac1{2} \int_\O |\grad_\Tt \xi(\btau^n)|^2 \d\x  - \frac{\|u'\|_\infty d m_\O \|\grad \tau_D\|_\infty}2.
\ee
Putting~\eqref{eq:T1}--\eqref{eq:T3} in~\eqref{eq:T123} ends the proof of Lemma~\ref{lem:wBV}.
\end{proof}

\begin{prop}\label{prop:existence}
Let $\btau^{n-1} \in \X_{\Tt, \rm int}$,  there exists a unique solution $\btau^n \in \X_{\Tt,\rm int}$ 
to the scheme~\eqref{eq:scheme}--\eqref{eq:FKsig}. Moreover, it satisfies $\tau_K^n \ge \tau_\star$ for all $K  \in \Tt$.
\end{prop}
\begin{proof}
The uniqueness of the solution was proven at Lemma~\ref{lem:unique} while the fact that $\btau^n \ge \tau_\star$
was the purpose of Lemma~\ref{lem:pos}. Therefore, it only remains to show the existence 
of a solution. 
It follows from Estimate~\eqref{eq:energy} that 
$$
\| \grad_\Tt \left(\xi(\btau^n) - \xi(\btau_D^n)\right) \|_{L^2(\O)^d}^2 \le 2 \frac{\Eee^{n-1}(\btau^{n-1})}{\dt^n} + 2 C_1 + 4 \|u'\|_\infty \| \grad \tau_D\|_\infty^2.
$$
Since $\xi$ is coercive~\eqref{eq:xi-coercive2}, there exists $C$ depending on 
the data (among which $\btau^{n-1}$, $\tau_D$, the mesh $\Tt$ and the time step $\dt^n$) such that 
$$
|\tau_K^n| \le C, \qquad \forall K \in \Tt.
$$
This estimate allows us to make use of a topological degree argument (see~\cite{LS34,Dei85,EGGH98}) 
to prove the existence of one solution to the scheme. 
\end{proof}

The convergence of the scheme can be proved following the path proposed in~\cite{EGH99}.
Enhanced convergence properties can be obtained thanks to the recent contribution~\cite{DE16}.
But this is not the goal of this paper.  We are mainly interested in the practical computation of the 
approximate solution at fixed discretization parameters. In particular, we focus on the behavior the 
Newton's method.

\section{About the Newton method}\label{sec:Newton}

The numerical scheme~\eqref{eq:scheme}--\eqref{eq:FKsig} amounts for all $n \in \{1,\dots, N\}$ to 
the nonlinear system 
\be\label{eq:syst-n2}
\Ff_n(\btau^n) = \left(f_K^n(\btau^n)\right)_{K\in\Tt}= \0, \qquad \text{with}\quad \Ff_n \in \Cc^2\left(\R^{\#\Tt};\R^{\#\Tt}\right), 
\ee
where 
\begin{multline*}
f_K(\btau) = (s(\tau_K) - s_{K}^{n-1}) \\
+ \frac{\dt^n}{m_K} \sum_{\sig \in \Ee_K} \left( m_\sig g_{K,\sig}^- (\lambda(s(\tau_K^n)) - \lambda(s(\tau_{K,\sig}^n))) + A_\sig \left(u(\tau_K^n) - u(\tau_{K,\sig}^n)\right)\right).
\end{multline*}
Assume that the Jacobian matrix $\bbJ_{\Ff_n}(\btau^n)$ of $\Ff_n$ at $\btau^n$ 
is not singular (this will be shown for nondegenerate parametrizations~\eqref{eq:param-nondeg}, cf. Proposition~\ref{prop:Jac}).
Approximating the solutions to the system~\eqref{eq:syst-n} with the Newton method 
consists in the construction of a sequence $\left(\btau^{n,k}\right)_{k\ge 0}$ defined by 
\be\label{eq:Newton}
\begin{cases}
\btau^{n,0} = \btau^{n-1}, \\
\btau^{n,k+1} = \btau^{n,k} - \left[\bbJ_{\Ff_n}(\btau^{n,k})\right]^{-1} \Ff_n(\btau^{n,k}).
\end{cases}
\ee
If the method converges, 
the solution $\btau^n$ is then defined as 
\be\label{eq:btau-n}
\btau^n = \lim_{k\to\infty} \btau^{n,k}.
\ee
Moreover, the convergence speed is asymptotically quadratic if $\Ff_n \in \Cc^2$, 
i.e., 
\be\label{eq:quadra}
\|\btau^{n,k} - \btau^n\| \le C \|\btau^{n,k-1} - \btau^{n}\|^2, \qquad \forall k \ge k_\star\; \text{large enough}
\ee
where, denoting by $\Vv(\btau^n)$ a neighborhood of $\btau^n$ in $\X_{\Tt}$, the quantity $C$ 
(as well as $k_\star$) depends on 
\be\label{eq:cond-quadra-loc}
\sup_{\btau \in \Vv(\btau^n)} \left\|\left[\bbJ_{\Ff_n}(\btau)\right]^{-1}\right\|
\qquad \text{and}\qquad 
\sup_{\btau \in \Vv(\btau^n)} \left\| D^2\Ff_n(\btau)\right\|.
\ee
Since $\btau^n$ is unknown, a sufficient condition to ensure that~\eqref{eq:quadra} holds for 
some $C>0$ is the following uniform non-degeneracy condition
\be\label{eq:UNDC}
\sup_{\btau \in \R^{\#\Tt}} \left\|\left[\bbJ_{\Ff_n}(\btau)\right]^{-1}\right\| <\infty.
\ee

We cite here a simplified version of the so-called Newton-Kantorovich theorem (see, e.g.,~\cite{Kanto48,Ortega68,OR70}).
We refer to~\cite{GT74} for a quantitative version of the theorem, and to~\cite{QS93} to a non-smooth 
version.
\begin{thm}[Newton-Kantorovich theorem]\label{thm:Kanto}
Assume that there exists two positive quantities $C_2$ and $C_3$ such that 
\be\label{eq:Kanto-1}
\sup_{\btau \in \X_\Tt} \left\|\bbJ_{\Ff_n}(\btau)\right\| \le C_2 
\quad \text{and}\quad 
\sup_{\btau \in \X_\Tt} \left\|\left[\bbJ_{\Ff_n}(\btau)\right]^{-1}\right\| \le C_3, 
\ee
then there exists $\rho>0$ such that 
$$
\|\btau^{n,0} - \btau^n\| \le \rho \quad \implies \quad \btau^{n,k} \underset{k\to\infty}{\longrightarrow} \btau^n.
$$
\end{thm}

Our strategy in the sequel is to prove that a non-degenerate parametrization (in the sense of~\eqref{eq:param-nondeg}) 
yields estimates~\eqref{eq:Kanto-1} in the subordinate matrix $1$-norm.

\begin{rem}\label{rem:non-smooth}
The assumption $\Ff \in \Cc^1$ can be relaxed. More precisely, Newton method can be extended 
to the case where $\Ff$ is merely semi-smooth~\cite{Clarke90} following the way proposed by 
Qi and Sun in~\cite{QS93}. Quadratic convergence is preserved in this nonsmooth case provided 
$\Ff_n$ is semi-smooth of order 1, cf.~\cite[Theorem 3.2]{QS93}. 
Our study can be extended to this more general case, but, for the sake
of simplicity, we have chosen to reduce our presentation to the classical smooth framework.
\end{rem}

\subsection{Some technical lemmas related to $M$-matrices}\label{ssec:Mmatrices}

Because of the elliptic degeneracy of the problem when $\tau\ge 0$, we cannot apply the 
the results of~\cite{Fuhrmann01} to get estimates on the Jacobian matrix $\bbJ_{\Ff_n}$. 
We need to introduce some technical material to circumvent this difficulty.
Let us first define that notion of $\delta$-transmissive path (see also~\cite{CG16_MCOM, CG_VAGNL})

\begin{Def}[$\delta$-transmissive path]
Let $\delta>0$, let $\bbA = \left(a_{ij}\right)_{1\le i,j\le N} \in \Mm_N(\R)$ and let $(i,j) \in \{1,\dots, N\}$. 
\begin{enumerate}
\item A row-wise $\delta$-transmissive path $\Pp(i,j)$ of length $L$ from $i$ to $j$ associated to the matrix $\bbA$ 
consists in a list $\{k_0, \dots, k_L\}$ with 
\begin{enumerate}[(i)]
\item $k_0 = i$, $k_L = j$, and $k_p \neq k_q$ if $p \neq q$; 
\item  for all $p \in \{0,\dots, L-1\}$, one has  $a_{k_p k_{p+1}} < - \delta$.
\end{enumerate}
\item The list  $\{k_0, \dots, k_L\}$ is a  column-wise $\delta$-transmissive path associated to the matrix $\bbA$ 
if it is a row-wise $\delta$-transmissive path associated to $\bbA^T$.
\end{enumerate}
\end{Def}

\begin{Def}[$(\delta,\Delta)$-M-matrices] \label{Def:dDM}
Let $\delta, \Delta>0$ be such that $\Delta > \delta$.  
\begin{enumerate}
\item 
A matrix $\bbA =\left(a_{ij}\right)_{1\le i,j\le N} \in \Mm_N(\R)$ is said to be a row-wise $(\delta,\Delta)$-M-matrix if 
\begin{enumerate}[(i)]
\item for all $i \in \{1,\dots, N\}$, one has $\delta \le a_{ii} \le \Delta$, $a_{ij} \le 0$ for all $j \neq i$ 
and $\sum_{j=1}^N a_{ij} \ge 0$;
\item the set $\Ii_\delta(\bbA) = \left\{i \in \{1,\dots, N\} \; \middle| \; \sum_{j=1}^N a_{ij} \ge \delta \right\}$ is 
not empty, and for all $i \in \Ii_\delta(\bbA)^c =  \{1,\dots, N\}\setminus  \Ii_\delta(\bbA)$, 
there exists a $\delta$-transmissive path $\Pp(i,j)$ with $j \in \Ii_\delta(\bbA)$.
\end{enumerate}
\item A matrix $\bbA =\left(a_{ij}\right)_{1\le i,j\le N} \in \Mm_N(\R)$ is said to be a column-wise $(\delta,\Delta)$-M-matrix 
if $\bbA^T$ is a row-wise $(\delta,\Delta)$-M-matrix.
\end{enumerate}
\end{Def}

It is well known that $M$-matrices are invertible matrices. The goal of Lemma~\ref{lem:row} and of Corollary~\ref{coro:column}
is to get uniform estimates on the inverse of a uniform M-matrix. Let us stress that the estimates we obtain are far from 
being optimal in the applications we have in mind, namely the Finite Volume discretization of Richards' equation.

\begin{lem}\label{lem:row}
Let $\bbA\in \Mm_N(\R)$ be a row-wise $(\delta,\Delta)$-M-matrix, 
then there exists $C$ depending only on $\delta, \Delta$ and $N$ such that 
$
\| \bbA^{-1}\|_\infty \le C.
$
\end{lem}
\begin{proof}
For the ease of reading, we denote $\|\cdot\|$ instead of $\|\cdot\|_\infty$, and $\Ii_\delta$ instead of $\Ii_\delta(\bbA)$.
The property $\| \bbA^{-1}\| = \frac1\alpha$ is equivalent to 
\be\label{eq:PF1}
\frac1{\|\bbA^{-1}\|}=\max_{\|\bz\|=1}\|\bbA \bz\| = \alpha.
\ee
Let $\bz = \left(z_i\right)_{1\le i \le N}\in \R^N$ with $\|\bz\|=1$. 
We assume, without loss of generality that there exists $i \in \{1,\dots N\}$ such that $z_i = 1$.
In the sequel, we denote by $\Ll$ the maximum length of a transmissive path, i.e., 
\be\label{eq:Ll}
\Ll := \max_{j \in \Ii_\delta(\bbA)^c} \min_{j\in \Ii_\delta(\bbA)}{\rm length} \left(\Pp(i,j)\right) \le N-1.
\ee

Assume first that $i \in \Ii_\delta$, then $\left(\bbA \bz\right)_i = \sum_j a_{ij} z_j \ge \delta$, 
hence $\|\bbA^{-1}\| \le \frac1\delta$.
Assume now that $i \in \Ii_\delta^c$, and let $\{k_p, \; 0 \le p \le L\}$ be a transmissive path with 
$k_0 = i$, $k_L \in \Ii_\delta$ and $L \le \Ll$. Denote by 
$$
c_p = \left( \left(\frac\Delta \delta \right)^p - 1\right) \frac1{\Delta - \delta}.
$$
Let us show by induction that 
\be\label{eq:HR}
z_{k_p} \ge 1- c_p \alpha \, \qquad \forall p \in \{0,\dots, L\}.
\ee
Since $z_{k_0}=1$, the relation~\eqref{eq:HR} holds for $p=0$. 
Now suppose that~\eqref{eq:HR} holds for some $p \in \{0,\dots, N-1\}$. 
One knows from~\eqref{eq:PF1} that 
$
\sum_{j=1}^N a_{k_pj} z_j \le \alpha, 
$
hence, using~\eqref{eq:HR}, that $\sum_{j=1}^N a_{k_pj}\ge 0$, and that 
$a_{k_p j} z_j \ge a_{k_p j}$ for all $j \notin \{p,p+1\}$, one gets that 
$$
-a_{k_pk_{p+1}} (1-z_{k_{p+1}})   \le \alpha (1+ c_p a_{k_pk_p}). 
$$
Since $a_{k_pk_p} \le \Delta$ and $a_{k_pk_{p+1}} \le -\delta$, this leads to 
$$
z_{k_{p+1}} \ge 1 - \alpha \frac{1+\Delta c_p}{\delta} = 1 - \alpha c_{p+1}, 
$$
so that the proof of~\eqref{eq:HR} is complete. 
As a consequence, one gets that $z_{k_L} \ge 1 - \alpha c_L$, and thus that 
$$
a_{k_Lk_L} (1- c_L \alpha) + \sum_{j \neq k_L} a_{k_Lj} \le \alpha.
$$
Using that $k_L \in \Ii_\delta$, we obtain that 
$
\left(\Delta c_L +1 \right)\alpha \ge \delta.
$
Therefore, $$\alpha \ge \frac1{c_{L+1}} \ge \frac1{c_{\Ll+1}}$$
since the length $L$ of the path $\{k_0, \dots k_L\}$ is bounded by $\Ll$ defined by~\eqref{eq:Ll}.
\end{proof}

\begin{coro}\label{coro:column}
Let $\bbA\in \Mm_N(\R)$ be a column-wise $(\delta,\Delta)$-M-matrix, 
then there exists $C$ depending only on $\delta, \Delta$ and $N$ such that 
$
\| \bbA^{-1}\|_1 \le C.
$
\end{coro}

\subsection{A uniform non-degeneracy result}\label{ssec:UND}

\begin{prop}\label{prop:Jac}
Assume that $s$ and $u$ satisfy he non-degeneracy condition~\eqref{eq:param-nondeg}, 
then there exist $C_2$ depending only on $\alpha^\star, \g, \|\lambda'\|_{\infty}, \dt^n$ 
and $\Tt$, and $C_3$ depending only on $\alpha_\star, \alpha^\star, \dt^n$ and $\Tt$ such that 
$$
\| \bbJ_{\Ff_n}(\btau) \|_1 \le C_2, \qquad \| \left[\bbJ_{\Ff_n}(\btau)\right]^{-1} \|_1 \le C_3, \qquad  \forall \btau \in \R^{\#\Tt}.
$$
\end{prop}
\begin{proof}
Let us first make the Jacobian matrix $\bbJ_{\Ff_n}(\btau) = \left(j_{KL}^n(\btau)\right)_{K,L \in \Tt}$ explicit:
\begin{align*}
j_{KK}^n(\btau) =& s'(\tau_K) + \frac{\dt^n}{m_K} \sum_{\sig \in \Ee_K} \left( m_\sig g_{K,\sig}^+ \lambda'(s(\tau_K)) s'(\tau_K) 
+ A_\sig {u'(\tau_K)}\right), \\
j_{LK}^n(\btau) =& -  \frac{\dt^n}{m_K} \left(m_\sig g_{K,\sig}^+ \lambda'(s(\tau_K)) s'(\tau_K) 
+ A_\sig {u'(\tau_K)}\right) \quad \text{where}\;\sig = K|L.
\end{align*}
Proving that $\| \bbJ_{\Ff_n} \|_1 \le C_2$ is easy since all the coordinates of $\bbJ_{\Ff_n}$ are uniformly 
bounded w.r.t. $\btau$ thanks to the upper bound~\eqref{eq:su-Lip} on $s'$ and $u'$.

Let us now prove that $\| \left[\bbJ_{\Ff_n}(\btau)\right]^{-1} \|_1 \le C_3$. 
Thanks to Corollary~\ref{coro:column}, it suffices to check that $\bbJ_{\Ff_n}$ is a column-wise $(\delta, \Delta)$-M-matrix 
where $\delta$ and $\Delta$ depend on the prescribed quantities. 
First, it is easy to check that 
\be\label{eq:(i)}
j_{LK}^n(\btau) \le 0 \; \text{if} \; L \neq K, \quad \sum_{L\in\Tt} j_{LK}^n(\btau) \ge 0, \quad \text{and}\quad0 \le \delta_1 \le j_{KK}^n(\btau) \le  \Delta, 
\ee
where 
\begin{align*}
\delta  =\;& \alpha_\star \min\left\{ 1 ; \dt^n \min_K \left(\frac{\min_{\sig \in \Ee_K} A_\sig}{m_K}\right) \right\}, 
\\
\Delta=\;&\alpha^\star \max_{K\in\Tt} \left(1 +\frac{\dt^n}{m_K} \sum_{\sig \in \Ee_K} 
(m_\sig g_{K,\sig}^+ \|\lambda'\|_\infty + A_\sig )  \right).
\end{align*}
Therefore, Condition (i) in Definition~\ref{Def:dDM} is fulfilled. 

It follows from the non-degeneracy condition~\eqref{eq:param-nondeg} that for all $\btau \in \X_{\Tt, \rm int}$ and 
all $K\in \Tt$, either $s'(\tau_K) \ge \alpha_\star$ or $u'(\tau_K) \ge \alpha_\star$. Let $K$ be such that 
$s'(\tau_K) \ge \alpha_\star$, then 
$$
\sum_{L\in\Tt} j_{LK}^n(\btau) \ge \alpha_\star \ge \delta, 
$$
whence $K \in \Ii_\delta(\bbJ_{\Ff_n}(\btau)^T)$.
On the other hand, if $u'(\tau_K) \ge \alpha_\star$ and $\Ee_{K,\rm ext} \neq \emptyset$, 
then 
$$
\sum_{L\in\Tt} j_{LK}^n(\btau) \ge \alpha_\star \frac{\dt^n}{m_K}  \sum_{\sig \in \Ee_{K,\rm ext}} A_\sig \ge \delta.
$$
As a consequence, if $K \notin  {\Ii_{\delta}(\bbJ_{\Ff_n}(\btau)^T)}$, one has 
necessarily that $u'(\tau_K) \ge \alpha_\star$ and $\Ee_{K,\rm ext} = \emptyset$. 
But in this case, 
$$
j_{LK}^n(\btau) \le - \alpha_\star \frac{\dt^n  A_\sig}{m_K} \le -\delta \quad \text{if} \; \sig = K|L \in \Ee_{K,\rm int}.
$$
The matrix $\bbD^n \in \R^{\#\Tt \times \#\Tt}$ defined by 
$$
\bbD_{KK}^n = \alpha_\star \frac{\dt^n}{m_K} \sum_{\sig \in \Ee_K} A_\sig, \qquad 
\bbD_{KL}^n = \begin{cases}
- \alpha_\star \frac{\dt^n}{m_K} A_\sig & \text{if}\; \sig = K|L\\
0 & \text{otherwise}
\end{cases}
$$
is irreducible and 
admits $\delta$-transmissive paths from $K$ to $L$ for all $(K,L) \in \Tt^2$.
This ensures that $\bbJ_{\Ff_n}(\btau)$ admits a $\delta$-transmissive path 
from any cell $K \in {\Ii_{\delta}(\bbJ_{\Ff_n}(\btau)^T)}^c$ to any cell 
$L \in {\Ii_{\delta}(\bbJ_{\Ff_n}(\btau)^T)}$. Therefore, Condition (ii) of Definition~\ref{Def:dDM} 
is fulfilled. $\bbJ_{\Ff_n}(\btau)^T$ is then a row-wise $(\delta,\Delta)$-M-matrix, thus 
$\bbJ_{\Ff_n}(\btau)$ is a column-wise $(\delta,\Delta)$-M-matrix.
\end{proof}

The following corollary is a straightforward compilation of Newton-Kantorovich theorem~\ref{thm:Kanto} 
with Proposition~\ref{prop:Jac}.
\begin{coro}[local convergence of Newton's method]\label{coro:local-conv}
There exists $\rho^n>0$ such that the Newton method~\eqref{eq:Newton} converges as soon as 
$\|\tau^{n,0} - \tau^n\|\le \rho^n$.
\end{coro}

\begin{rem}[global convergence for small enough time steps]
The radius $\rho^n$ appearing in Corollary~\ref{coro:local-conv} can be estimated thanks to~\cite{GT74}
as soon as the second order derivatives $s''$ and $u''$ of $s$ and $u$ are uniformly bounded.
It appears in a fairly natural way that $\rho^n$ is a non-decreasing function of $\dt^n$. 
Then choosing $\dt^n$ small enough, the variation between the time steps $t^{n-1}$ and $t^n$ is small and 
$$\| \btau^{n-1} - \btau^n\|_1 = \| \btau^{n,0} - \btau^n\|_1 \le \rho^n.$$
Therefore, the convergence of the Newton method is ensured if one uses an adaptive time-step algorithm 
(as for instance in~\cite{CG_VAGNL}).
\end{rem}

\subsection{Control of the error induced by the inexact Newton procedure}\label{ssec:inexact}
Assume that $s(\btau^{n-1})$ is exactly known, then 
the exact solution $\btau^n$ is obtained as the limit of $\left(\btau^{n,k}\right)_k$. 
Computing the exact value of $\btau^n$ is impossible and
a convenient criterion must be adopted in order to stop the iterative procedure.
This yields errors that accumulate along time. The goal of this section is to quantify 
the error induced by the inexact resolution of the nonlinear system.

\subsubsection{One step error estimates}\label{sssec:onestep}
In this section, we assume that $s(\btau^{n-1})$ is exact, and we want to quantify the error 
corresponding to one single iteration. In what follows, 
we consider the following residual based stopping criterion:
\be\label{eq:stopping}
\text{stop the iterative procedure~\eqref{eq:Newton} if}\; \left\|\Ff_n(\btau^{n,k})\right\|_1 = \sum_{K\in\Tt} \left|f_K(\tau_K^{n,k})\right| \le \eps \dt^n 
\ee
for some prescribed tolerance $\eps>0$. We denote by $\btau^n_\eps=\btau^{n,k}$ when~\eqref{eq:stopping} is fulfilled and the loop is stopped.
Then the non-degeneracy~\eqref{eq:Kanto-1} of $[\bbJ_{\Ff_n}]^{-1}$ provides directly the following error estimate:
$$
\|\btau^{n}_\eps - \btau^n\|_1 \le C_3 \eps \dt^n.
$$
More than in the variable $\tau$, whose physical sense is unclear, we are interested in evaluating 
the error on the reconstructed saturation profile.
Thanks to the Lipschitz continuity of $s$ ---recall the non-degeneracy assumption~\eqref{eq:param-nondeg}---, we have 
\be\label{eq:err-tau}
\|\pi_\Tt s(\btau^{n}_\eps) - \pi_\Tt s(\btau^n)\|_{L^1(\O)} \le C_4 \eps \dt^n, 
\ee
where $C_4 = \alpha^\star (\max_K m_K)  C_3$.

\subsubsection{Quantification of the error accumulation}

In the previous paragraph, it was assumed that $s(\btau^{n-1})$ was exactly known. 
In practice, one may consider that we know $s(\btau^0)$ exactly, but merely the approximation of 
$s(\btau^{n}_\eps)$ obtained after stopping the Newton iterative procedure after a finite number of iterations. 
Denote by $\left(s(\btau^n)\right)_{1\le n \le N}$ the iterated exact solutions to the scheme~\eqref{eq:scheme}--\eqref{eq:FKsig}, 
and by $\left(s(\btau^{n}_\eps)\right)_{1\le n \le N}$ the iterated inexact solutions obtained \emph{via} the Newton method 
with the stopping criterion~\eqref{eq:stopping}. Define $\Ff_{n,\eps}: \X_{\Tt, \rm int} \to \X_{\Tt, \rm int}$ by 
\be\label{eq:Ffneps}
\Ff_{n,\eps}(\btau) = \Ff_{n}(\btau) + s(\btau^{n-1})- s(\btau_\eps^{n-1}).
\ee
In particular, one has $\bbJ_{\Ff_{n,\eps}}(\btau) = \bbJ_{\Ff_{n}}(\btau)$ for all $\btau$ and the results of \S\ref{ssec:UND} 
still hold for $\Ff_{n,\eps}$ instead of $\Ff_n$. The inexact Newton method then writes
\begin{enumerate}
\item {\tt Initialization:} define $\btau^0_\eps = \btau^0$ by~\eqref{eq:btau0}.
\item {\tt From $t^{n-1}$ to $t^n$:} 
\begin{enumerate}
\item set $\btau^{n}_\eps = \btau^{n-1}_\eps$
\item iterate Newton's algorithm until $\|\Ff_{n,\eps}(\btau^{n}_\eps)\|_1 \le \eps \dt^n$
\end{enumerate}
\end{enumerate}
Then, as claimed by the following statement, the $L^1(\O)$ error on the 
reconstructed saturation growth at most linearly with time. In particular, no 
exponential amplification of the error occurs in this context. 
\begin{prop}\label{prop:err-inexact}
Let $\left(\btau^n\right)_n$ be the exact solution to the scheme \eqref{eq:scheme}--\eqref{eq:FKsig}, and 
let $\left(\btau^n_\eps\right)_n$  be the approximate solution computed by the inexact Newton method,  
then
\be\label{eq:err-inexact}
\|\pi_\Tt s(\btau^n_\eps) - \pi_\Tt s(\btau^{n})\|_{L^1(\O)} \le C_4 \eps t^n, \qquad \forall n \in \{0,\dots, N\}, 
\ee
where $C_4$ is introduced was~\eqref{eq:err-tau}.
\end{prop}
\begin{proof}
We perform the proof by induction. Estimate~\eqref{eq:err-inexact}  clearly holds for $n=0$ since the initialization is exact. 
Now assume that it holds for $n-1$. Denote by $\wt \btau^n_\eps$ the exact solution of the system 
$\Ff_{n,\eps}(\wt \btau^n_\eps) = {\bf 0}$. Then it follows from the discussion carried out in~\S\ref{sssec:onestep} that 
$$
\|\pi_\Tt s(\btau^n_\eps) - \pi_\Tt s(\wt \btau^n_\eps) \|_{L^1(\O)} \le C_4 \eps \dt^n.
$$
On the other hand, applying Lemma~\ref{lem:contract}, one gets that 
$$
\|\pi_\Tt s(\wt \btau^n_\eps) - \pi_\Tt s(\btau^n) \|_{L^1(\O)} \le \|\pi_\Tt s(\btau^{n-1}_\eps) - \pi_\Tt s(\btau^{n-1}) \|_{L^1(\O)} \le C_4 \eps t^{n-1}.
$$
One concludes thanks to the triangle inequality.
\end{proof}

\section{Numerical validation of the approach}\label{sec:num}

In the previous section, we have shown that as long as the parametrization $\tau \mapsto (s(\tau), u(\tau))$ satisfies 
the condition \eqref{eq:param-nondeg}, the Newton's method applied to the problem \eqref{eq:syst-n2} exhibits local convergence. Moreover, due to Proposition~\ref{prop:err-inexact}, the error resulting from an inexact Newton's method can be efficiently controlled. 
Remark that the constant $C_4$ in \eqref{eq:err-inexact} depends on the nonlinearities $S$ and $\eta$ only 
through the quantities $\alpha_\star$ and $\alpha^\star$. 
Therefore, the estimate~\eqref{eq:err-inexact} is robust with respect to the hydrodynamic properties of the soil. We illustrate this fact by numerical experiments presented below. As an example of mobility/capillary pressure relations we consider a popular Brooks-Corey model~\cite{BC64} for which we compare the efficiency of Newton's method resulting from the parametrization $u(\tau) = \tau$ and the one satisfying \eqref{eq:param-nondeg} with $\alpha_\star = \alpha^\star = 1$.
\smallskip

Let $p_b< 0$ and $\beta > 0$, in Brooks-Corey model the saturation and the mobility functions are given by
$$
S(p) = \left\{ \begin{array}{lll} \displaystyle \left( \frac{p}{p_b} \right)^{-\beta} & \text{if}\;p < p_b,\\
1 & \text{if}\; p \geq p_b,
\end{array}
\right.
$$
and 
$$
\lambda(s) = s^{3+\frac{2}{\beta}},
$$
providing in terms of Kirchhoff transform
$$
\tilde{S}(u) = \left\{ \begin{array}{lll} \displaystyle \left( \frac{u}{u_b} \right)^{\frac{1}{\eta}}
&\text{if}\; u < \displaystyle u_b,\\
1 &\text{if}\; u \geq u_b,
\end{array}
\right.
$$
with $\displaystyle \eta = \beta + 3 + \frac{1}{\beta}$ and $\displaystyle u_b = -\frac{p_b }{\beta\eta}$.

Remark that the parametrization based on $u(\tau) = \tau$ and $s(\tau) = \tilde{S}(u(\tau))$, referred as $u-$formulation, do not satisfy \eqref{eq:param-nondeg} since the derivative of $\tilde{S}(u)$ is singular at $u=0$. As an alternative we consider the parametrization $\tau \mapsto (s(\tau), u(\tau))$  defined by the equation $\max( s'(\tau), u'(\tau) ) = 1$ and the condition $s(0) = 0$, to which we refer as $\tau-$formulation. 
We obtain the following explicit formulas
$$
s(\tau) = \left\{ \begin{array}{lll} 
\displaystyle \tau &\text{if}\; \tau < \tau_{\star},\\
\displaystyle S( \tau - \tau_\star + u_b \tau_\star^\eta ) &\text{if}\; \tau \geq \tau_{\star},
\end{array}
\right.
$$
and 
$$
u(\tau) = \left\{ \begin{array}{lll} 
\displaystyle u_b \tau^\eta &\text{if}\; \tau < \tau_{\star},\\
\displaystyle \tau - \tau_\star + u_b \tau_\star^\eta &\text{if}\; \tau \geq \tau_{\star},
\end{array}
\right.
$$
with $\displaystyle \tau_\star =  \min \left( \left( {\eta u_b} \right)^{\frac{1}{1-\eta} }, 1 \right)$.

\subsection{First test case}

We consider a bidimensional porous domain $\O = (0,1)\times (0,1)$, which is initially very dry with $s_0(\x) = 10^{-6}$ in $\O$. The water is injected at the pressure $p_D = 1$ through the portion of an upper boundary $\G_D = \{ (x_1, x_2) ~|~ x_1 \in (0, 0.3), x_2 = 1 \}$. The gravity vector is given by $\bold{g} = - \nabla x_2$ and a zero flux boundary condition is prescribed on $\partial \O \setminus \G_D$. The computations are performed on two quasi-uniform space discretizations of $\O$ composed of 396 and 1521 Vorono{\"i} cells referred as Mesh 1 and Mesh 2. The final time is set to $T = 0.7$ and the time step is equal to $0.01$. Figure \ref{fig_case_1_snap} shows, for $\beta = 4$ and $p_b = -10^{-2}$, the distribution of saturation and generalized pressure $u$ at different times. The results are visualized on triangular Delaunay mesh dual to Mesh 1.

\begin{figure}[htb]
    \centering
    \begin{subfigure}[b]{0.3\textwidth}
        \centering
        {\includegraphics[width=\textwidth]{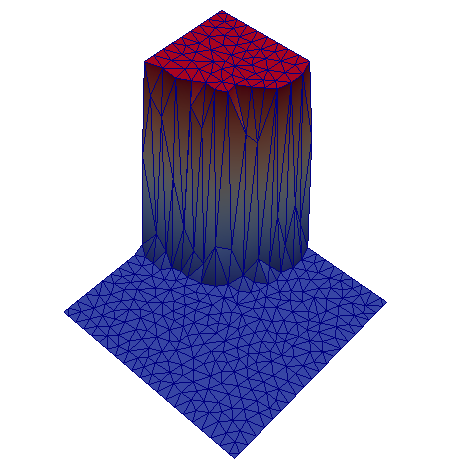}}
        \centering
        \caption{$s$ at $t = 0.1$}
        \label{fig_case_1_a}
    \end{subfigure}
    \hfill
    \begin{subfigure}[b]{0.3\textwidth}
        \centering
        {\includegraphics[width=\textwidth]{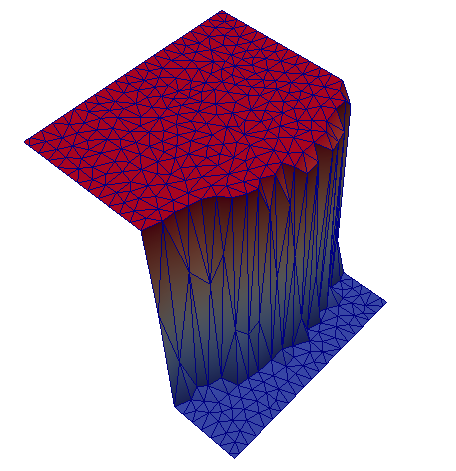}}  
        \centering
        \caption{$s$ at $t = 0.5$}
        \label{fig_case_1_b}
    \end{subfigure}
    \hfill
    \begin{subfigure}[b]{0.3\textwidth}
        \centering
        {\includegraphics[width=\textwidth]{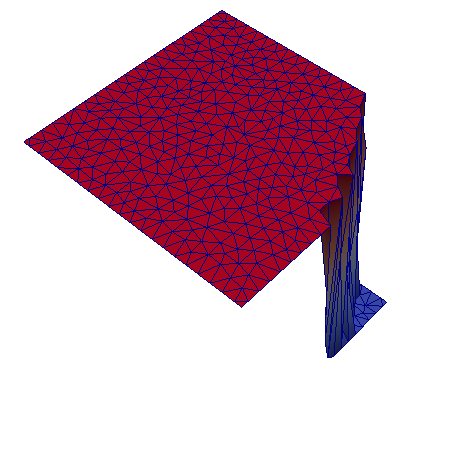}}  
        \centering
        \caption{$s$ at $t = 0.7$}
        \label{fig_case_1_c}
    \end{subfigure}    
    \hfill   
    \begin{subfigure}[b]{0.3\textwidth}
        \centering
        {\includegraphics[width=\textwidth]{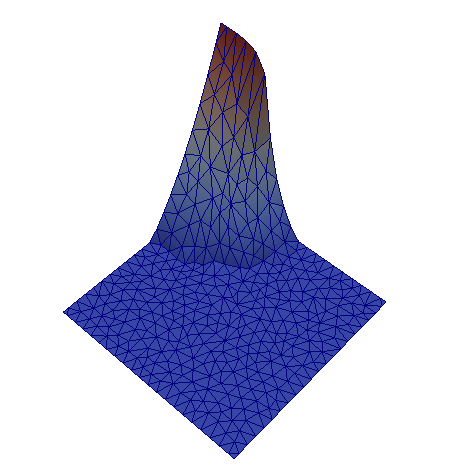}}
        \centering
        \caption{$u$ at $t = 0.1$}
        \label{fig_case_1_a}
    \end{subfigure}
    \hfill
    \begin{subfigure}[b]{0.3\textwidth}
        \centering
        {\includegraphics[width=\textwidth]{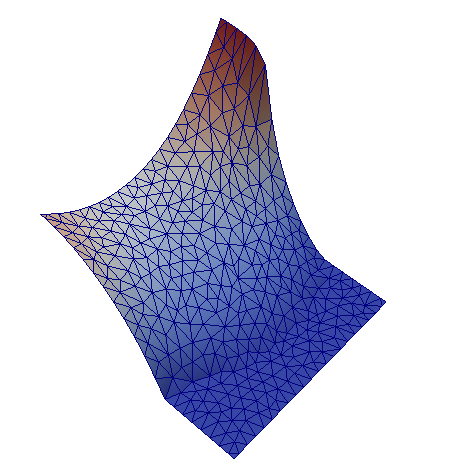}}  
        \centering
        \caption{$u$ at $t = 0.5$}
        \label{fig_case_1_b}
    \end{subfigure}
    \hfill
    \begin{subfigure}[b]{0.3\textwidth}
        \centering
        {\includegraphics[width=\textwidth]{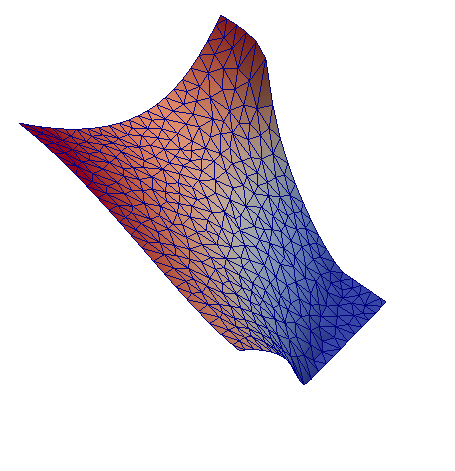}}  
        \centering
        \caption{$u$ at $t = 0.7$}
        \label{fig_case_1_c}
    \end{subfigure}    
    \hfill   
    \caption{Snapshots of the reference solution at different times.}
    \label{fig_case_1_snap}
\end{figure}

In order to challenge the robustness of both formulations, we set $p_b = -10^{-2}$ and we let the 
parameter $\beta$ take values in the set $\{1, 2, 4, 8, 16\}$. For each value of $\beta$ we compute, using $\tau-$formulation and tolerance $\eps_{ref} = 10^{-12}$, the reference solution denoted by $\left(\btau^n_{\beta} \right)_{n\in \{1,\ldots, N\}} \in \X_\Dd$. Then, for both formulations and for the values of $\eps \in \{10^{-2}, 10^{-4}, 10^{-6}, 10^{-8}, 10^{-10}, 10^{-12}\}$, we perform the calculations measuring the total number of Newton's iteration and the deviation, in the discrete $L^{\infty}(L^1)$ norm, of the ``observable'' variables $u$ and $s$ from the reference solution. 

For a given value of $\beta$ and of the tolerance $\eps$, we denote by $\left(\overline{\btau}^n_{\beta,\eps} \right)_{n\in \{1,\ldots, N\}} \in \X_\Dd$ and $\left(\btau^n_{\beta,\eps} \right)_{n\in \{1,\ldots, N\}} \in X_\Dd$ the approximate solution of \eqref{eq:syst-n2} obtained using the $u-$formulation and $\tau-$formulation respectively. The error produced by inexact Newton's method is measured by the quantities
$$
\overline{err}^{u}_{\beta,\eps} = \frac{ \| \pi_\Dd \overline{\btau}^n_{\beta,\eps}  - \pi_\Dd u( \btau^n_\beta )  \|_{L^{\infty}(0,T;L^1(\O))}} { \| \pi_\Dd u( \btau^n_\beta )  \|_{L^{\infty}(0,T;L^1(\O))} }
$$
and 
$$
\overline{err}^{s}_{\beta,\eps} = \frac{ \| \pi_\Dd \tilde{S} (  \overline{\btau}^n_{\beta,\eps} )  - \pi_\Dd s( \btau^n_\beta )  \|_{L^{\infty}(0,T;L^1(\O))}} { \| \pi_\Dd s( \btau^n_\beta )  \|_{L^{\infty}(0,T;L^1(\O))} }
$$
for $u-$formulation, and 
$$
err^{u}_{\beta,\eps} = \frac{ \| \pi_\Dd u( \btau^n_{\beta,\eps} ) - \pi_\Dd u( \btau^n_\beta )  \|_{L^{\infty}(0,T;L^1(\O))}} { \| \pi_\Dd u( \btau^n_\beta )  \|_{L^{\infty}(0,T;L^1(\O))} }
$$
and 
$$
err^{s}_{\beta,\eps} = \frac{ \| \pi_\Dd s( \btau^n_{\beta,\eps} ) - \pi_\Dd s( \btau^n_\beta )  \|_{L^{\infty}(0,T;L^1(\O))}} { \| \pi_\Dd s( \btau^n_\beta )  \|_{L^{\infty}(0,T;L^1(\O))} }
$$
for $\tau-$formulation.

\begin{figure}
    \centering
    \begin{subfigure}[b]{0.49\textwidth}
        \centering
        {\includegraphics[width=\textwidth]{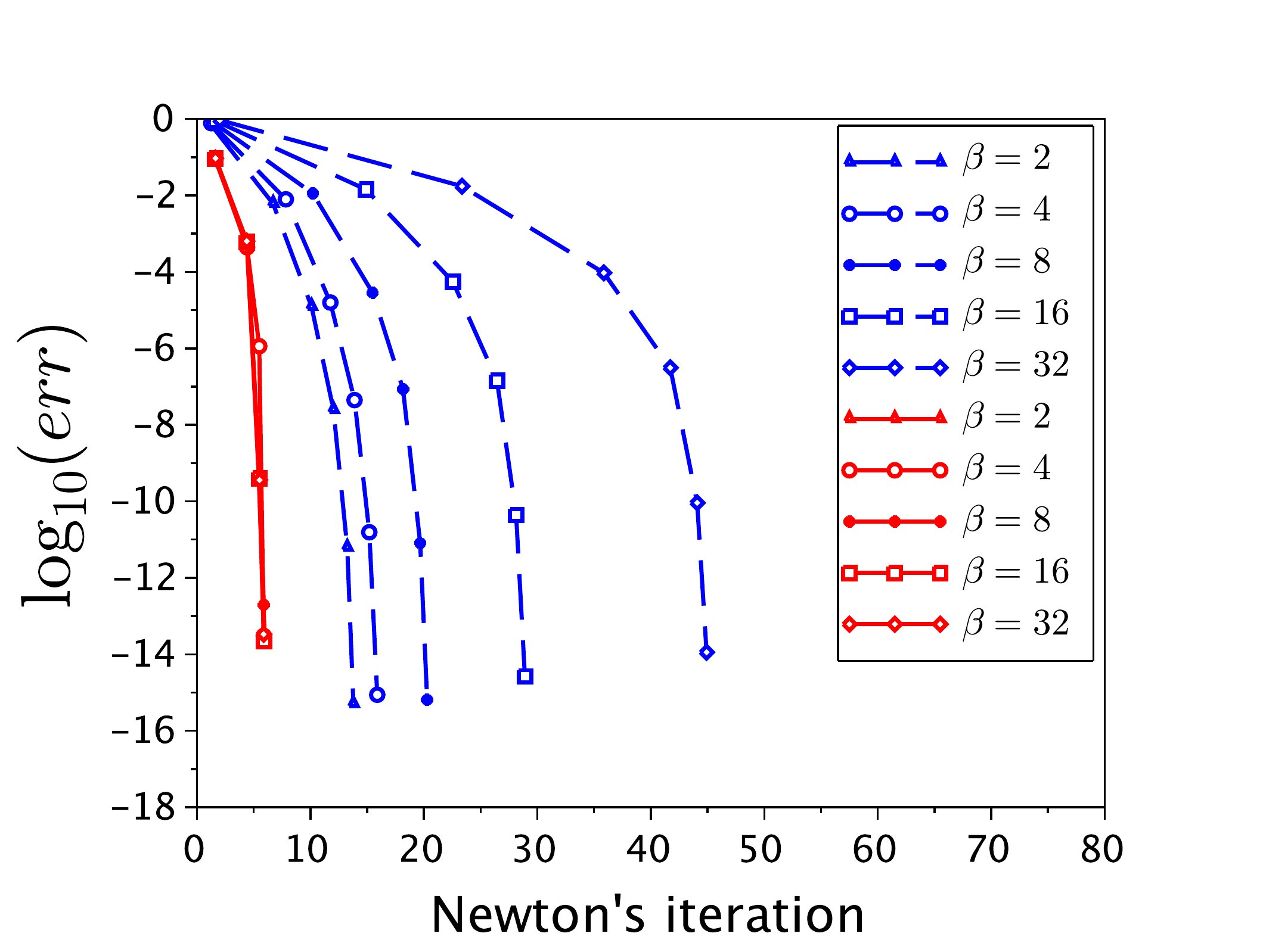}}
        \centering
        \caption{$err^{s}_{\beta,\eps}$ and $\overline{err}^{s}_{\beta,\eps}$ using Mesh 1.}
        \label{fig_case_1_a}
    \end{subfigure}
    \hfill
    \begin{subfigure}[b]{0.49\textwidth}
        \centering
        {\includegraphics[width=\textwidth]{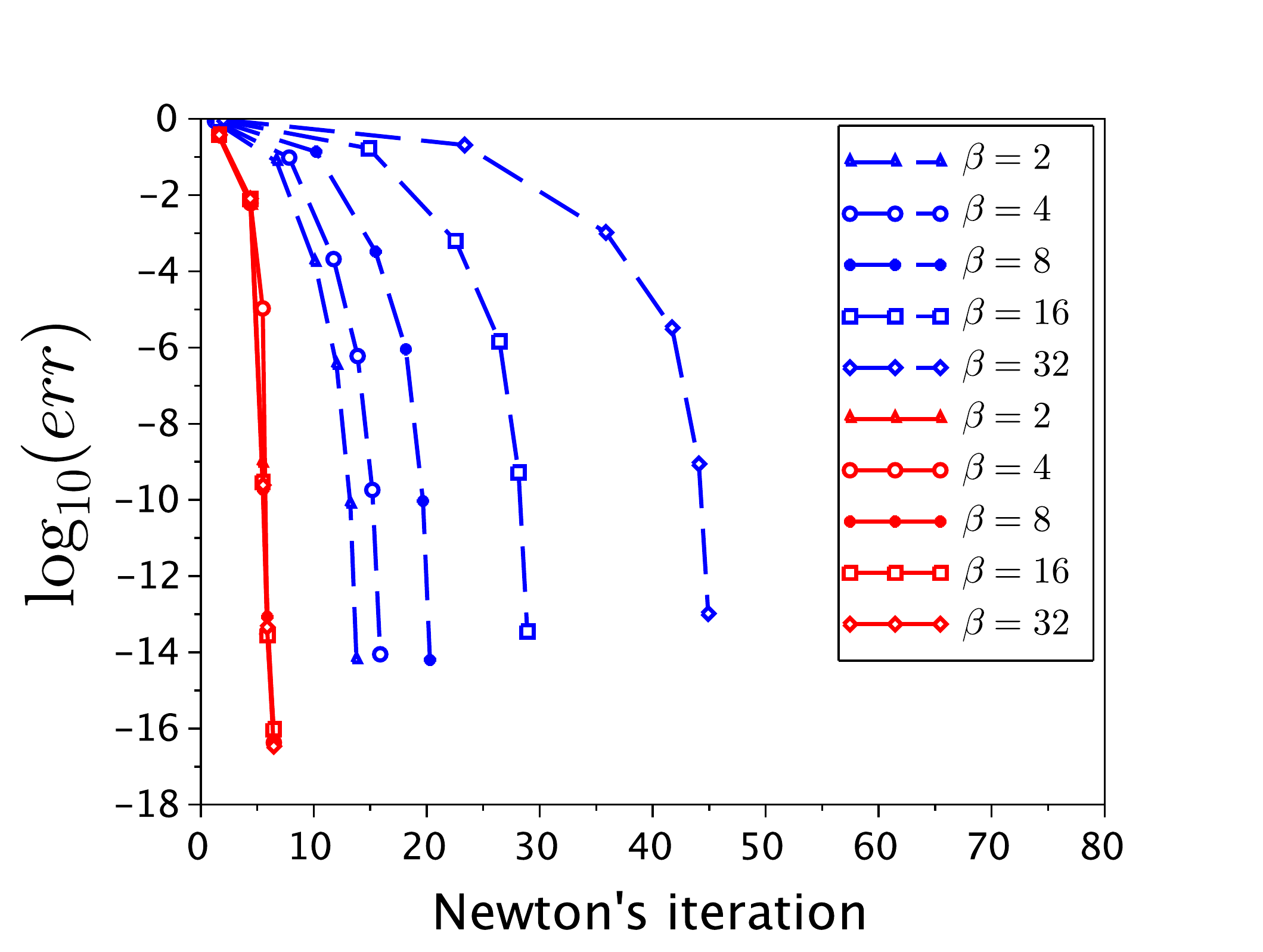}}  
        \centering
        \caption{$err^{u}_{\beta,\eps}$ and $\overline{err}^{u}_{\beta,\eps}$ using Mesh 1.}
        \label{fig_case_1_b}
    \end{subfigure}
    \hfill
    \begin{subfigure}[b]{0.49\textwidth}
        \centering
        {\includegraphics[width=\textwidth]{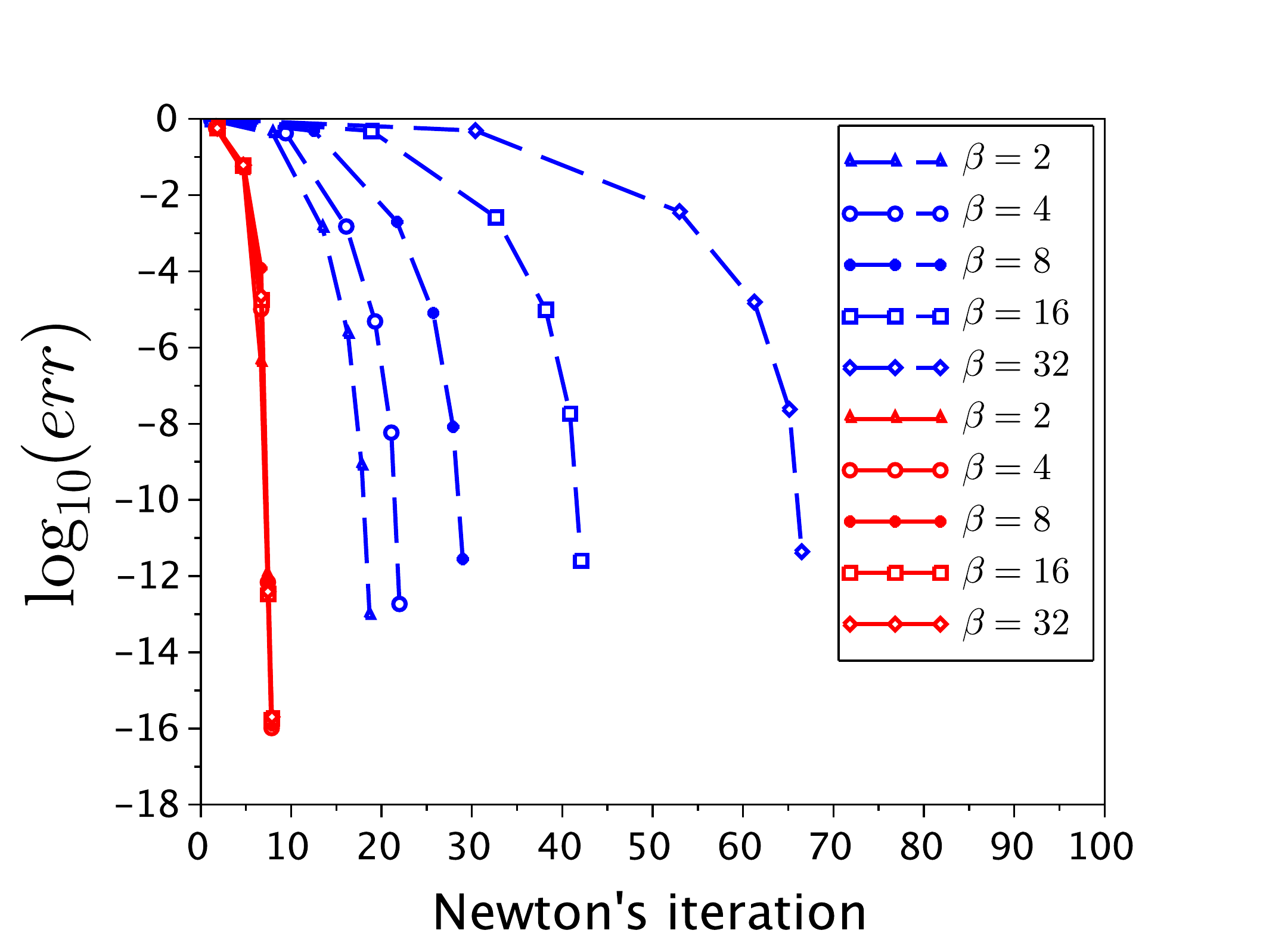}}  
        \centering
        \caption{$err^{s}_{\beta,\eps}$ and $\overline{err}^{s}_{\beta,\eps}$ using Mesh 2.}
        \label{fig_case_1_c}
    \end{subfigure}    
    \hfill
    \begin{subfigure}[b]{0.49\textwidth}
        \centering
        {\includegraphics[width=\textwidth]{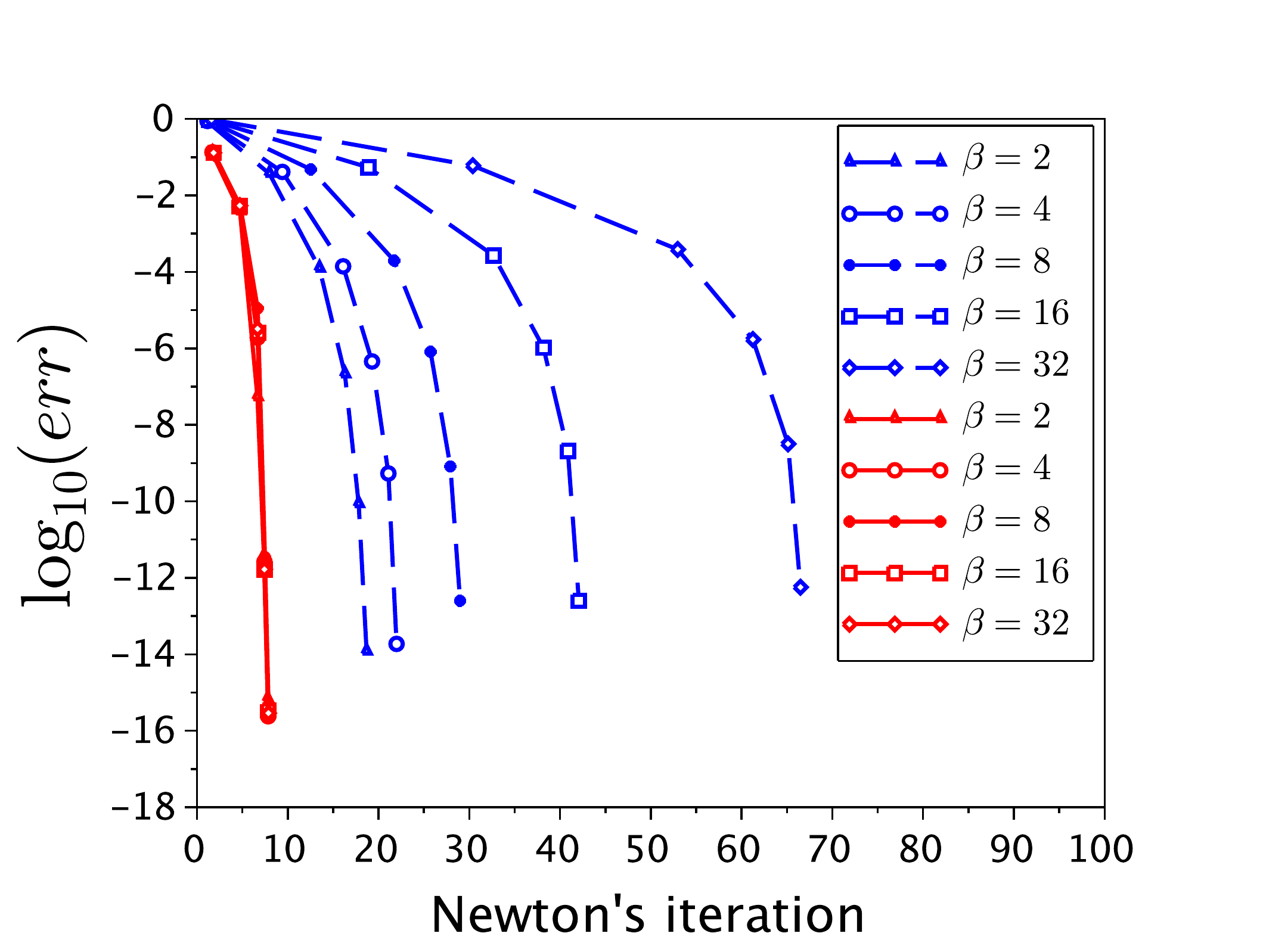}}  
        \centering
        \caption{$err^{u}_{\beta,\eps}$ and $\overline{err}^{u}_{\beta,\eps}$ using Mesh 2.}
        \label{fig_case_1_d}
    \end{subfigure}    
    \hfill
    \caption{Relative error as the function of the average number of Newton's iterations per time step using $\tau-$formulation (red solid lines) and $u-$formulation (blue dashed lines).}
    \label{fig_case_1}
\end{figure}
Figure \ref{fig_case_1} exhibits, for Meshes 1 and 2, the behavior of the relative error $err^\xi_{\beta,\eps}$ and $\overline{err}^\xi_{\beta,\eps}$, $\xi = u, s$ as the function of an average number of iterations per time step required by Newton's method in order to converge up to the given tolerance $\eps$. We observe that in order to achieve the same precision $u-$formulation require a much larger number of iterations then $\tau-$formulation. Moreover the number of Newton's iterations, for $u-$formulation, increases with $\beta$, whereas $\tau$-formulation remains robust with respect to this parameter. The contrast in the efficiency of two formulations is amplified as the mesh is refined.

\subsection{Second test case}

The goal of this test case is to give a numerical evidence that the inexact Newton's method applied to the 
$u$-formulation produces large errors due to troubles in the conservation of mass. 
To do so, we prescribe a zero-flux condition on the whole boundary of $\O = (0,1)\times (0,1)$, 
whereas the initial saturation satisfies 
$$
s_0 = \left\{\begin{array}{llll}
0.5 & \mbox{in} & \O', \\
10^{-6} & \mbox{in} & \O \setminus \O'
\end{array}\right. 
$$ 
with $\O' = \{ (x_1, x_2) ~|~ x_1 < 0.5 \;\text{and}\; x_2 > 0.5 \}$. We set $\beta = 4$ and $p_b = -10^{-2}$. The effects of gravity neglected so that the flow is only driven by diffusion. Figure \ref{fig_case_2_snap} exhibits, for different times, the saturation field associated with the reference solution, which is computed using $\tau-$formulation and the tolerance $\eps_{ref} = 10^{-16}$. Since the flow is very slow the large time steps are needed, we set $T = 10^5$ and $\Delta t = 10^3$. Computations are performed using Mesh 1.
\begin{figure}[h!]
    \centering
    \begin{subfigure}[b]{0.45\textwidth}
        \centering
        {\includegraphics[width=\textwidth]{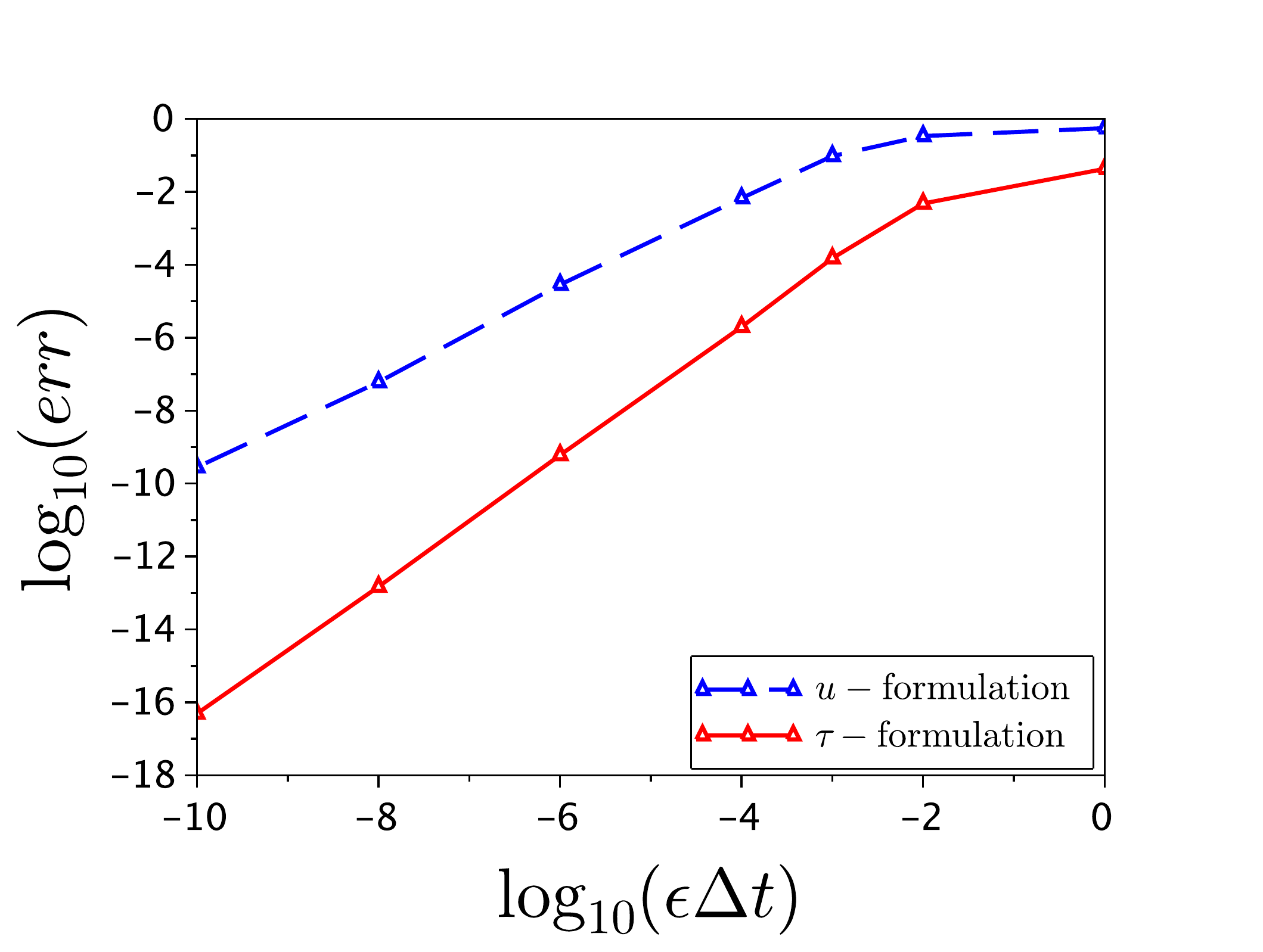}}  
        \centering
        \label{fig_case_2_errs_a}
    \end{subfigure}
    \hfill
    \begin{subfigure}[b]{0.45\textwidth}
        \centering
        {\includegraphics[width=\textwidth]{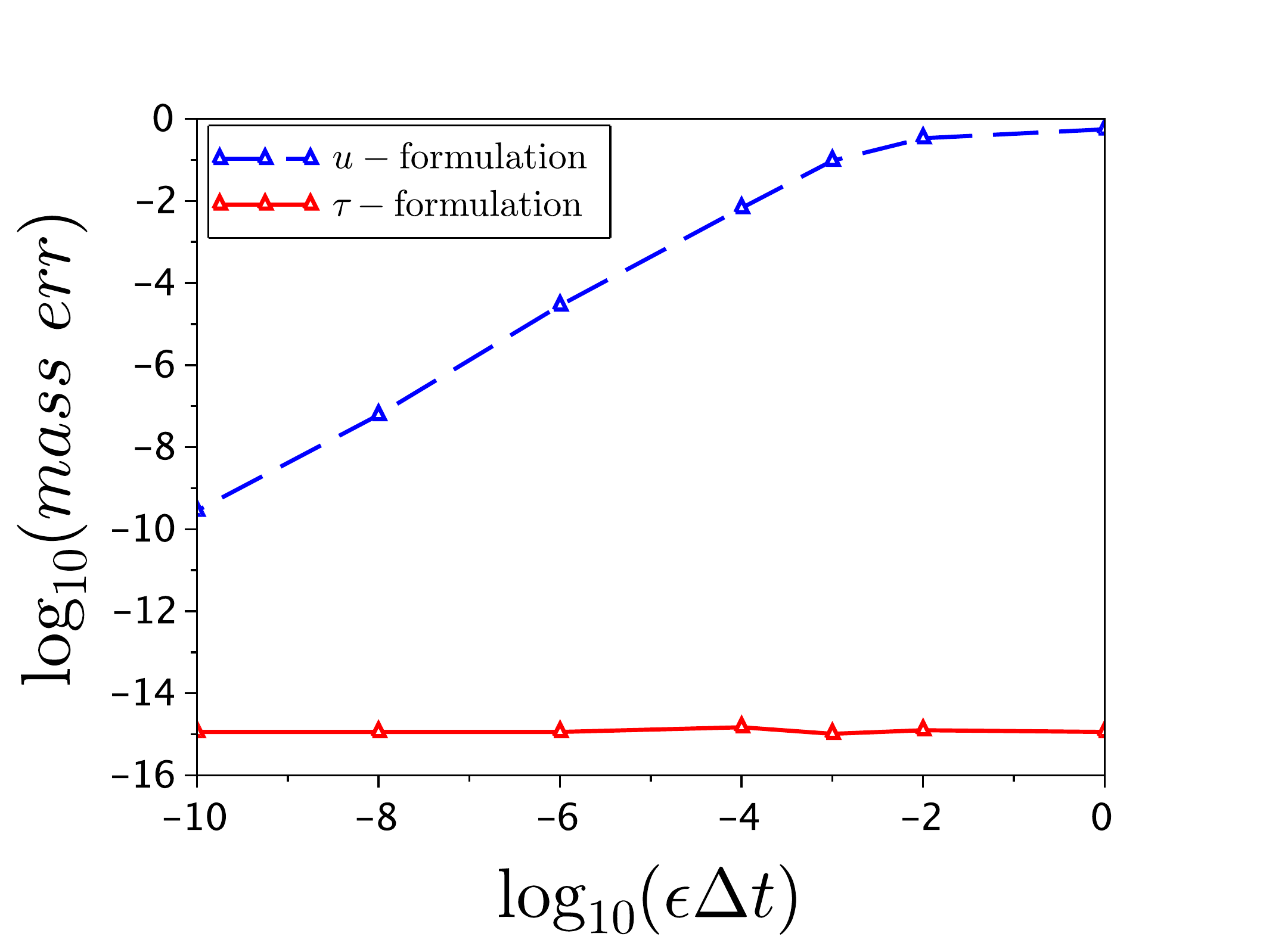}}  
        \centering
        \label{fig_case_2_errs_b}
    \end{subfigure}    
    \caption{At left $err^{s}_{\beta,\eps}$ (red solid lines) and $\overline{err}^{s}_{\beta,\eps}$ (blue dashed lines) as the function of $\eps$, at right the relative mass conservation error as the function of $\eps$.}
    \label{fig_case_2_errs}
\end{figure}
Let $M =\int_\O s_0 {\rm d}\x$, we define the relative mass conservation error by
$$
\overline{mass~err}^{s}_{\beta,\eps} = \frac{1}{M} \max_{n\in\{1,\ldots,N\}} \left| \int_\O \pi_\Dd \tilde{S} (  \overline{\btau}^n_{\beta,\eps} ) {\rm d}\x  - M \right| 
$$
and
$$
mass~err^{s}_{\beta,\eps} = \frac{1}{M} \max_{n\in\{1,\ldots,N\}}\left| \int_\O \pi_\Dd s (  \btau^n_{\beta,\eps} ) {\rm d}\x  - M \right|.
$$ 
Figure \ref{fig_case_2_errs} exhibits the $L^\infty(L^1)$ relative saturation error $err^{s}_{\beta,\eps}$, $\overline{err}^{s}_{\beta,\eps}$, and the relative mass conservation error $mass~err^{s}_{\beta,\eps}$ and $\overline{mass~err}^{s}_{\beta,\eps} $ as the functions of $\eps$. As one can see 
the $L^\infty(L^1)$ error produced by $u-$formulation is dominated by the mass conservation error. Remark that even for the rather small values $\eps=10^{-6}$ or $\eps=10^{-7}$, the error produced by $u-$formulation is still significant (see Figures \ref{fig_case_2_cut}).  
In contrast $\tau-$formulation leads to mush smaller errors and, for any value of $\eps$, conservatives the mass up to a precision of order $10^{-15}$. 

The very high accuracy for mass conservation observed with the $\tau-$formulation can be explained as follows.  With the values of the parameters $\beta, p_b$ we have chosen, one has $s(\tau) = \tau$ for $\tau \in [0,1]$ and hence for all $\tau \in \O \times (0,T)$ in view of initial and boundary conditions.
In addition, in view of \eqref{eq:FKsig} we have 
$$\sum_{K\in \Tt} \sum_{\sigma\in\Ee_K} F_{K,\sigma}(\btau^n) = 0, \qquad \forall \btau^n \in \bbX_\Dd.$$
Therefore, at each step of inexact Newton's method, the flux contribution globally offset, hence we have  
$$
\sum_{K\in\Tt} m_K \left( s(\tau^{n,k}_{\eps,K}) + s'(\tau^{n,k}_{\eps, K})( \tau^{n,k+1}_{\eps, K} - \tau^{n,k}_{\eps,K} ) - s(\tau^{n-1}_{\eps,K})  \right) = 0.
$$
Since $\tau \mapsto s(\tau)$ is linear, one has $s(\tau^{n,k+1}_{\eps,K}) = s(\tau^{n,k}_{\eps,K}) + s'(\tau^{n,k}_K)( \tau^{n,k+1}_{\eps, K} - \tau^{n,k}_{\eps,K} )$, which implies that the mass is exactly conserved (assuming that linear algebraic computations are exact) at each iteration of Newton's method.

\begin{figure}[h!]
    \centering
    \begin{subfigure}[b]{0.3\textwidth}
        \centering
        {\includegraphics[width=\textwidth]{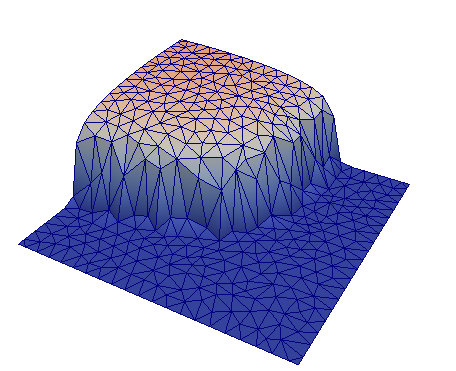}}
        \centering
        \caption{$t = 5~10^3$}
        \label{fig_case_1_a}
    \end{subfigure}
    \hfill
    \begin{subfigure}[b]{0.3\textwidth}
        \centering
        {\includegraphics[width=\textwidth]{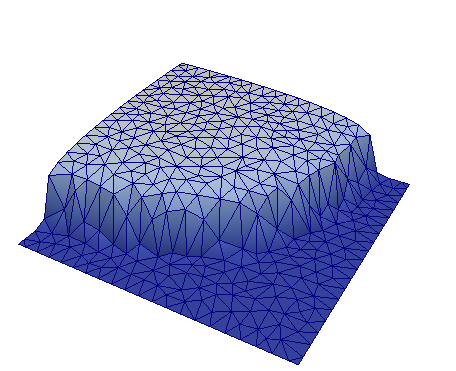}}  
        \centering
        \caption{$t = 50~10^3$}
        \label{fig_case_1_b}
    \end{subfigure}
    \hfill
    \begin{subfigure}[b]{0.3\textwidth}
        \centering
        {\includegraphics[width=\textwidth]{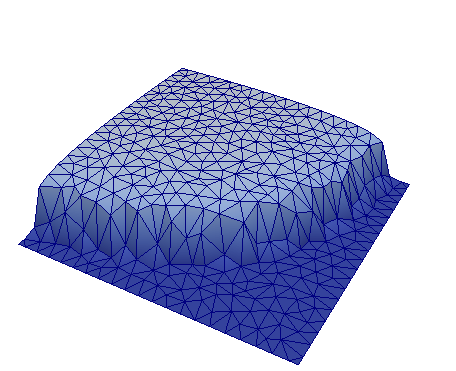}}  
        \centering
        \caption{$t = 100~10^3$}
        \label{fig_case_1_c}
    \end{subfigure}    
    \hfill   
    \caption{Saturation field of the reference solution at different times.}
    \label{fig_case_2_snap}
\end{figure}
\begin{figure}[h!]
    \centering
    \begin{subfigure}[b]{0.3\textwidth}
        \centering
        {\includegraphics[width=\textwidth]{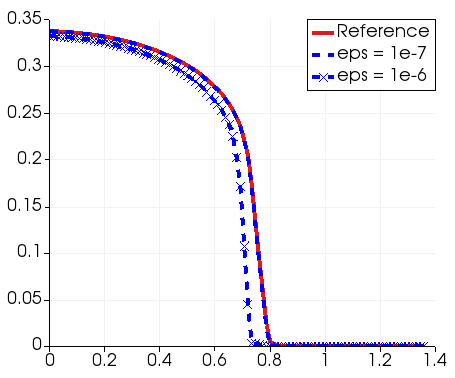}}
        \centering
        \caption{$t = 5~10^3$}
        \label{fig_case_1_a}
    \end{subfigure}
    \hfill
    \begin{subfigure}[b]{0.3\textwidth}
        \centering
        {\includegraphics[width=\textwidth]{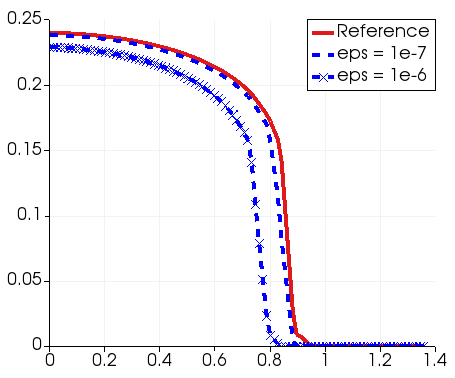}}  
        \centering
        \caption{$t = 50~10^3$}
        \label{fig_case_1_b}
    \end{subfigure}
    \hfill
    \begin{subfigure}[b]{0.3\textwidth}
        \centering
        {\includegraphics[width=\textwidth]{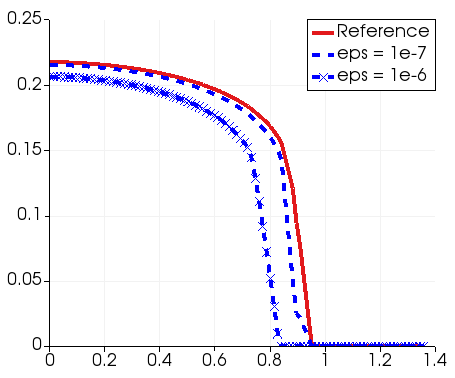}}  
        \centering
        \caption{$t = 100~10^3$}
        \label{fig_case_1_c}
    \end{subfigure}    
    \hfill   
    \caption{Saturation profile of the reference solution along the line $y = 1 - x$ compared to the saturation profiles of the approximate solutions computed using $u$-formulation and $\eps = 10^{-8}$ and $10^{-9}$.}
    \label{fig_case_2_cut}
\end{figure}

\def\ocirc#1{\ifmmode\setbox0=\hbox{$#1$}\dimen0=\ht0 \advance\dimen0
  by1pt\rlap{\hbox to\wd0{\hss\raise\dimen0
  \hbox{\hskip.2em$\scriptscriptstyle\circ$}\hss}}#1\else {\accent"17 #1}\fi}

\bigskip
\bigskip

\small

\noindent Konstantin \textsc{Brenner} \\
Laboratoire Jean-Alexandre Dieudonn\'e, Universit\'e de Nice Sophia Antipolis, \\
Team Coffee INRIA Sophia Antipolis M\'editerran\'ee, \\
06108 Nice Cedex 02, France.\\
\href{konstantin.brenner@unice.fr}{\tt konstantin.brenner@unice.fr}

\bigskip

\noindent Cl{\'e}ment \textsc{Canc{\`e}s} \\
Team Rapsodi INRIA Lille - Nord Europe, \\
40, avenue Halley, 59650 Villeneuve d'Ascq, France.\\
\href{clement.cances@inria.fr}{\tt clement.cances@inria.fr}

\end{document}